\documentclass[11pt,reqno]{amsart}

\usepackage{amsmath,amssymb}
\usepackage{microtype}
\usepackage[hidelinks]{hyperref}
\allowdisplaybreaks
\numberwithin{equation}{section}

\newtheorem{theorem}{Theorem}[section]
\newtheorem{proposition}[theorem]{Proposition}
\newtheorem{lemma}[theorem]{Lemma}
\newtheorem{corollary}[theorem]{Corollary}
\theoremstyle{remark}

\newcommand{\E}{\mathbb E}
\newcommand{\dd}{\,\mathrm d}
\newcommand{\e}{\mathrm e}
\newcommand{\MP}{\mathrm{MP}}
\newcommand{\unit}{\mathrm i}
\newcommand{\DReal}{\Delta^{\mathbb R}}
\renewcommand\le{\leqslant}
\renewcommand\ge{\geqslant}

\title[Average Singular Values and the Rectangularity Transition]{Dimension Monotonicity in Laguerre Ensembles II: Average Singular Values and the Rectangularity Transition in the Orthogonal Case}
\author{Ondrej Hutn\'{i}k}
\address{Institute of Mathematics, Faculty of Science, Pavol Jozef \v{S}af\'{a}rik University in Ko\v{s}ice, Jesenn\'{a} 5, 040 01 Ko\v{s}ice, Slovakia}
\email{ondrej.hutnik@upjs.sk}
\hypersetup{pdftitle={Dimension Monotonicity in Laguerre Ensembles II: Average Singular Values and the Rectangularity Transition in the Orthogonal Case},pdfauthor={Ondrej Hutnik}}

\begin{document}
\begin{abstract}
We study the normalized half moment
$\alpha_{\mathbb R}^{(\lambda)}(N)$ of the size-$N$ Laguerre orthogonal
ensemble for real shape $\lambda\ge0$.  At integer shape this is the expected
average singular value of an $N\times(N+\lambda)$ real Gaussian matrix.  The
square mean increases with the dimension, whereas every real shape
$\lambda\ge1$ decreases.  Between these two regimes the decrement is strictly
increasing in $\lambda$, and hence has a unique zero in $(0,1)$ for every
$N$.  There is also a unique crossing of the Marchenko--Pastur
limit.  Both thresholds converge to
$\lambda_*=1-\pi/4$, and their first corrections show that they separate on
the scale $(\log N)/N$.  The proof starts from an exact decomposition of the
real half moment into its complex counterpart and a positive orthogonal correction.  Recent
unitary estimates take care of the complex term.  An Abel completion,
together with a Laguerre connection formula, turns the orthogonal correction
into a positive diagonal series; the square case, the regime $\lambda\ge1$,
and the transition can then all be read from this same series.
\end{abstract}
\maketitle

\section{Introduction}\label{sec:intro}

\subsection{The model and the question}
For real $\lambda\ge0$, let the size-$N$ Laguerre orthogonal ensemble have
ordered eigenvalue density on
\[
 \mathcal W_N=\{0<x_1<\cdots<x_N\}
\]
proportional to
\begin{equation}\label{eq:LOE-jpd}
 \prod_{i=1}^N x_i^{(\lambda-1)/2}\e^{-x_i/2}
 \prod_{1\le i<j\le N}(x_j-x_i).
\end{equation}
This is the $\beta=1$ Laguerre ensemble with continuous shape parameter; see
\cite[Chapter~3]{Forrester} and the tridiagonal model of
Dumitriu--Edelman~\cite{DumitriuEdelman}.  When $\lambda\in\mathbb N_0$, it
is the eigenvalue law of
$X_{N,\lambda}^{\mathbb R}(X_{N,\lambda}^{\mathbb R})^T$ for an
$N\times(N+\lambda)$ matrix of independent $\mathcal N(0,1)$ entries.
Thus $\lambda$ has the same meaning as in Paper~I---the excess number of
columns over rows; the different power of $x$ in the joint density is only the
usual $\beta=1$ versus $\beta=2$ Laguerre parametrization.  For $(x_1,\ldots,x_N)$ distributed according to
\eqref{eq:LOE-jpd}, set
\[
 \alpha_{\mathbb R}^{(\lambda)}(N)
 :=N^{-3/2}\E\!\left[\sum_{j=1}^N\sqrt{x_j}\right].
\]
At integer shape this is exactly
$N^{-3/2}\E\|X_{N,\lambda}^{\mathbb R}\|_*$.  Following Paper~I, write
\[
 C_\lambda(N):=C_{1/2,\lambda}(N),\qquad
 \Gamma_{N,\lambda}:=\Gamma_{N,1/2,\lambda}
 =C_\lambda(N)-C_\lambda(N+1)
\]
for the analogous unitary half moment and its decrement.  Define the real
decrement by
\[
 \DReal_N(\lambda)
 :=\alpha_{\mathbb R}^{(\lambda)}(N)
   -\alpha_{\mathbb R}^{(\lambda)}(N+1).
\]
Thus, both papers use the same decrement orientation: a positive decrement means
that the normalized quantity decreases when the dimension is increased.  Both
$\alpha_{\mathbb R}^{(\lambda)}(N)$ and $C_\lambda(N)$ tend to $8/(3\pi)$
when the shape remains bounded. What matters here is the finite-dimensional direction of approach: \textit{Is the mean above or below its limiting value, and does one more dimension move it upward or downward?}

The square case already shows that the answer depends on the symmetry class.
For complex Gaussian matrices the mean decreases with $N$, whereas in the
real square case it increases.  The complex question arose in the Gaussian
rounding problem of Bandeira--Kennedy--Singer~\cite{BandeiraKennedySinger},
and the square complex and square real monotonicity statements were proved in
the first arXiv versions of Papers~I and II, respectively~\cite{HutnikComplexOriginal,HutnikRealOriginal}.  The
natural question is what happens between the square case and the first
rectangular case.  We show that the answer is a genuine transition in the
continuous Laguerre shape parameter, not merely a distinction between zero
and one excess column.  We refer to this continuous change of direction
between the square and rectangular regimes as the rectangularity transition.

\subsection{Two endpoint regimes}
The two endpoint regimes already point in opposite directions.  At
$\lambda=0$ the real mean moves upward with the dimension, while once the
shape reaches one it moves downward.  We begin with these two statements.

\begin{theorem}[Square endpoint]\label{thm:square-real}
For every $N\ge1$,
\[
 \alpha_{\mathbb R}^{(0)}(N+1)-\alpha_{\mathbb R}^{(0)}(N)
 >\frac1{160N^2}.
\]
Equivalently,
\[
 \DReal_N(0)<-\frac1{160N^2}.
\]
Consequently, the square orthogonal mean increases strictly to $8/(3\pi)$.
\end{theorem}

\begin{theorem}[The rectangular regime]\label{thm:main}
For every real $\lambda\ge1$ and every $N\ge1$,
\[
 \alpha_{\mathbb R}^{(\lambda)}(N+1)
 <\alpha_{\mathbb R}^{(\lambda)}(N).
\]
Consequently,
$\alpha_{\mathbb R}^{(\lambda)}(N)\searrow8/(3\pi)$ for every fixed
$\lambda\ge1$.
\end{theorem}

For integer $\lambda$, Theorem~\ref{thm:main} is the usual rectangular
Gaussian statement.  The continuous-shape formulation shows more: for each
fixed dimension the change of direction must occur somewhere between $0$ and
$1$.  The numerical constant $1/160$ in
Theorem~\ref{thm:square-real} is not meant to be optimal.  Its role is simply
to make the square mechanism quantitative: a single positive diagonal of
the orthogonal correction already dominates the whole square unitary
decrement.

\subsection{The transition}
Put
\[
 \lambda_*:=1-\frac\pi4.
\]
The key finite-$N$ fact is that the decrement is strictly increasing throughout
the strip.

\begin{theorem}[Monotonicity in the transition strip]
\label{thm:global-parameter-monotonicity}
For every $N\ge1$,
\[
 \partial_\lambda\DReal_N(\lambda)>0,
 \qquad 0\le\lambda\le1.
\]
\end{theorem}

The transition constant already appears at first order.  The unitary
decrement contributes $2\lambda/(\pi N^2)$, whereas the orthogonal correction
contributes the shape-independent mass $2/\pi-1/2$.  Their balance is exactly
$\lambda_*=1-\pi/4$.  Keeping the next term reveals how the finite-dimensional
thresholds approach this value.

\begin{theorem}[Two-term transition asymptotics]
\label{thm:transition-asymptotics}
For every fixed $\Lambda<\infty$, uniformly for $0\le\lambda\le\Lambda$,
\begin{align}
 \DReal_N(\lambda)
 &=\frac2\pi(\lambda-\lambda_*)\frac1{N^2}
 +\frac{3-4\lambda^2}{8\pi}\frac{\log N}{N^3}
 +O_\Lambda(N^{-3}),\label{eq:real-decrement-refined}\\
 \alpha_{\mathbb R}^{(\lambda)}(N)-\frac8{3\pi}
 &=\frac2\pi(\lambda-\lambda_*)\frac1N
 +\frac{3-4\lambda^2}{16\pi}\frac{\log N}{N^2}
 +O_\Lambda(N^{-2}).\label{eq:real-level-refined}
\end{align}
Consequently, for every fixed $\lambda\ne\lambda_*$, both signs eventually
agree with the sign of $\lambda-\lambda_*$.  At $\lambda=\lambda_*$ the
first-order terms vanish, but the logarithmic terms are positive.  Hence
$N\mapsto\alpha_{\mathbb R}^{(\lambda_*)}(N)$ is eventually above the
Marchenko--Pastur limit and decreases toward it.
\end{theorem}

For each fixed $N$, the level $\lambda\mapsto
\alpha_{\mathbb R}^{(\lambda)}(N)$ is also strictly increasing.  The endpoint
theorems and Theorem~\ref{thm:global-parameter-monotonicity} therefore give
two natural crossing points.

\begin{corollary}[Finite-dimensional crossings]\label{cor:crossings}
For every $N\ge1$ there are unique
$\lambda_N^{\mathrm{dec}},\lambda_N^{\mathrm{lev}}\in(0,1)$ such that
\[
 \DReal_N(\lambda_N^{\mathrm{dec}})=0,
 \qquad
 \alpha_{\mathbb R}^{(\lambda_N^{\mathrm{lev}})}(N)=\frac8{3\pi}.
\]
The decrement zero is simple, and
\begin{align}
 \lambda_N^{\mathrm{dec}}
 &=\lambda_*
 -\frac{3-4\lambda_*^2}{16}\frac{\log N}{N}
 +O(N^{-1}),\label{eq:lambdaN-convergence}\\
 \lambda_N^{\mathrm{lev}}
 &=\lambda_*
 -\frac{3-4\lambda_*^2}{32}\frac{\log N}{N}
 +O(N^{-1}).\label{eq:lambdaN-level-convergence}
\end{align}
Thus, $\lambda_N^{\mathrm{dec}}<\lambda_N^{\mathrm{lev}}<\lambda_*$ for all sufficiently large $N$.  Between the
two crossings the mean is still below the Marchenko--Pastur limit, but the
next dimension step is already downward.
\end{corollary}

Therefore, the ordering of the two crossings is the reverse of the convex unitary
transition in Paper~I: here the intermediate regime lies below the limiting
level while already moving downward.

\subsection{Idea of the proof and relation to earlier work}
The reason the two endpoint regimes can be treated by one argument is the
exact decomposition
\begin{equation}\label{eq:intro-decomposition}
 \alpha_{\mathbb R}^{(\lambda)}(N)
 =C_\lambda(N)-\Xi_N^{(\lambda)},
 \qquad \Xi_N^{(\lambda)}>0.
\end{equation}
The complex term is supplied by recent unitary work.  At square shape we use
the sharp non-asymptotic estimate of Abreu--Patil~\cite{AbreuPatil}; the
shape-dependent estimates come from Paper~I.  Those
estimates rest on the Laguerre moment recurrence of Cunden et
al.~\cite{CundenEtAl} and are closely related to the recurrence approach of
Baslingker--Dan~\cite{BaslingkerDan}.  What remains is the genuinely
orthogonal term $\Xi_N^{(\lambda)}$.  After an Abel completion and a Laguerre
connection formula, it becomes a positive diagonal series.  The square case, the half-line $\lambda\ge1$,
and the strip $0\le\lambda\le1$ are then three different comparisons of
the same object, rather than three unrelated proofs.

The square endpoint was first proved in the first arXiv version of the present
work~\cite{HutnikRealOriginal}, with the weaker explicit bound $1/(1000N^2)$.
The present paper grew out of that argument, but the proof has since changed
substantially.  Abreu--Patil now supply the square unitary bound, while the
shape-dependent unitary bounds come from Paper~I.  More importantly, the square
theorem is only the left endpoint of the
continuous-shape theory developed here.  The main extension is the transition
itself: both finite-$N$ crossings are unique, and their logarithmic
displacement from $\lambda_*$ can be computed.

This order-one sign transition is different from the shrinking-shape balance
at the half moment in Paper~I.  There the square correction is logarithmic, and
a shape of order $(\log N)/N$ is already large enough to affect the leading
finite-size correction.  In the orthogonal problem the extra correction is
already of order $N^{-1}$ at the level and $N^{-2}$ at the decrement, so the
balance takes place at an order-one shape.

\subsection{Organization}
Section~\ref{sec:orthogonal-correction} derives the real--complex identity and
the positive diagonal representation of the orthogonal correction.
Section~\ref{sec:finite-comparisons} proves the square and $\lambda\ge1$
monotonicity results and the monotonicity of the decrement on
$0\le\lambda\le1$.  Section~\ref{sec:asymptotics} identifies the two-term transition law and the
two finite-dimensional crossings.

\section{The orthogonal correction}\label{sec:orthogonal-correction}

The decomposition \eqref{eq:intro-decomposition} shifts the problem from the
real ensemble to the correction $\Xi_N^{(\lambda)}$.  The aim of this
section is to expose the sign structure of that correction.  We begin with
two basic facts about the continuous shape parameter, then write the
real--complex identity in the normalization used here, and finally convert
the orthogonal term into a positive diagonal series.

\subsection{Two basic facts about the shape parameter}

Two properties of the continuous Laguerre law will be used later to identify
the level crossing.  They are independent of the real--complex correction
formula, so it is useful to settle them first.

\begin{proposition}[Shape monotonicity and limiting level]
\label{prop:real-shape-basics}
For every $N\ge1$, the map
$\lambda\mapsto\alpha_{\mathbb R}^{(\lambda)}(N)$ is strictly increasing on
$[0,\infty)$.  Moreover,
\[
 \alpha_{\mathbb R}^{(\lambda)}(N)\longrightarrow\frac8{3\pi}
\]
as $N\to\infty$, locally uniformly for $\lambda\ge0$.
\end{proposition}

\begin{proof}
Let $f_{N,\lambda}$ be the normalized density in \eqref{eq:LOE-jpd}, extended
by zero outside $\mathcal W_N$.  In this proof,
$\operatorname{Cov}_\lambda$ denotes covariance with respect to this law.
Write
\[
 H(x)=\sum_{i=1}^N\sqrt{x_i},\qquad
 \ell(x)=\sum_{i=1}^N\log x_i.
\]
Differentiating the normalized integral is justified locally uniformly in
$\lambda\ge0$ by the hard-edge and exponential-tail integrability, and gives
\[
 \partial_\lambda\alpha_{\mathbb R}^{(\lambda)}(N)
 =\frac{1}{2N^{3/2}}\operatorname{Cov}_\lambda(H,\ell).
\]
The density $f_{N,\lambda}$ is MTP$_2$.  Indeed, the chamber $\mathcal W_N$
is a sublattice, the one-variable factors in \eqref{eq:LOE-jpd} are modular,
and for a Vandermonde factor $h(u,v)=v-u$ the only nontrivial ordering reduces,
after interchanging the two points if necessary, to
$x_i\le y_i<y_j\le x_j$, where
\[
 (y_j-x_i)(x_j-y_i)-(x_j-x_i)(y_j-y_i)
 =(y_i-x_i)(x_j-y_j)\ge0.
\]
The association theorem of Karlin--Rinott~\cite{KarlinRinott}, with truncation
for the unbounded functions, therefore gives
$\operatorname{Cov}_\lambda(\sqrt{x_i},\log x_j)\ge0$ for every $i,j$.
The diagonal terms are strictly positive.  If $X_i'$ is an independent copy
of the nondegenerate coordinate $X_i$, then
\[
 2\operatorname{Cov}_\lambda(\sqrt{X_i},\log X_i)
 =\E\!\left[(\sqrt{X_i}-\sqrt{X_i'})(\log X_i-\log X_i')\right]>0.
\]
Thus $\operatorname{Cov}_\lambda(H,\ell)>0$, proving strict shape
monotonicity.

For the limiting level, begin with an integer $m\ge0$ and let $\mu_N$ be the
empirical eigenvalue measure of
$N^{-1}X_{N,m}^{\mathbb R}(X_{N,m}^{\mathbb R})^T$.  Since
$(N+m)/N\to1$, the Marchenko--Pastur theorem~\cite{MarchenkoPastur} gives
$\mu_N\Rightarrow\mu_{\MP}$ almost surely.  The square-root moment passes to
the limit by truncation, because
\[
 0\le \sqrt x-(\sqrt x\wedge\sqrt M)\le M^{-1/2}x,
 \qquad
 \E\left[\int x\,\dd\mu_N(x)\right]=\frac{N+m}{N}.
\]
It follows that
\[
 \alpha_{\mathbb R}^{(m)}(N)\longrightarrow
 \int_0^4\sqrt x\,\mu_{\MP}(\dd x)
 =\frac1{2\pi}\int_0^4\sqrt{4-x}\,\dd x
 =\frac8{3\pi}.
\]
For a bounded interval $0\le\lambda\le\Lambda$, choose
$M=\lceil\Lambda\rceil$.  The monotonicity already proved gives
\[
 \alpha_{\mathbb R}^{(0)}(N)
 \le \alpha_{\mathbb R}^{(\lambda)}(N)
 \le \alpha_{\mathbb R}^{(M)}(N).
\]
Both endpoints tend to $8/(3\pi)$, proving uniform convergence on bounded
shape intervals.
\end{proof}

\subsection{Separating the real and complex parts}

The next step is to isolate the orthogonal contribution.  The underlying
real--complex identity is standard, but its factors of two depend on the
normalization, so we record the rescaling explicitly.  Let
$p_{N,\lambda}^{\mathbb R}$ and $p_{N,\lambda}^{\mathbb C}$ be the normalized
one-point densities of the continuous-shape $\beta=1$ and $\beta=2$
Laguerre ensembles in the normalizations of Section~\ref{sec:intro}.  With
this convention,
\[
 \alpha_{\mathbb R}^{(\lambda)}(N)
 =\frac1{\sqrt N}\int_0^\infty x^{1/2}
 p_{N,\lambda}^{\mathbb R}(x)\dd x,
 \qquad
 C_\lambda(N)
 =\frac1{\sqrt N}\int_0^\infty x^{1/2}
 p_{N,\lambda}^{\mathbb C}(x)\dd x.
\]
Let $\epsilon_N=N\bmod2$, and define
\small{\begin{align}
 w_{m,\epsilon_N}^{(\lambda)}
 &:=\frac{2^{1-\lambda}\Gamma(m+1-\epsilon_N/2)}
 {\Gamma(m+(\lambda+3)/2-\epsilon_N/2)},
 \qquad m\ge0,\notag\\
 \Phi_{1,N}^{(\lambda)}(x)
 &:=x^\lambda\e^{-x}
 \sum_{m=0}^{(N+\epsilon_N-2)/2}
 w_{m,\epsilon_N}^{(\lambda)}
 L_{2m+1-\epsilon_N}^{(\lambda)}(x),\label{eq:Phi1}\\
 \Phi_{2,N}^{(\lambda)}(x)
 &:={\left(x/2\right)^{(\lambda-1)/2}\e^{-x/2}}
 \left[
 (1-\epsilon_N)\frac{2\Gamma((\lambda+1)/2,x/2)}
 {\Gamma((\lambda+1)/2)}+2\epsilon_N-1
 \right].\label{eq:Phi2}
\end{align}}\normalsize
The finite-$N$ LOE one-point density with a continuous Laguerre exponent is
standard; see, for example, \cite[Eq.~(7.154)]{Forrester}.  We use the
equivalent form of Livan--Vivo~\cite[Eq.~(16)]{LivanVivo}, whose parameter
$\nu$ appears in the $\beta=1$ weight with exponent $(\nu-1)/2$
\cite[Eq.~(1)]{LivanVivo}.  Taking $\nu=\lambda$ therefore reproduces
\eqref{eq:LOE-jpd}.  This is a continuous Laguerre shape parameter, not a
continuation of the matrix dimension: integer $\lambda$ recovers the real
Wishart model, while the Dumitriu--Edelman $\beta$-Laguerre model realizes
all real $\lambda\ge0$~\cite{DumitriuEdelman}.

\begin{proposition}[Continuous-shape orthogonal correction]\label{prop:general-correction}
For every $N\ge1$ and every real $\lambda\ge0$,
\begin{equation}\label{eq:density-general}
 p_{N,\lambda}^{\mathbb R}(x)
 =p_{N,\lambda}^{\mathbb C}(x)
 -\frac1{2N}
 \frac{\Gamma((N+1)/2)}{\Gamma((N+\lambda)/2)}
 L_{N-1}^{(\lambda)}(x)
 \left\{\Phi_{1,N}^{(\lambda)}(x)
 -\Phi_{2,N}^{(\lambda)}(x)\right\}.
\end{equation}
Consequently,
\begin{equation}\label{eq:alpha-general-correction}
 \alpha_{\mathbb R}^{(\lambda)}(N)
 =C_\lambda(N)-\Xi_N^{(\lambda)},
 \qquad
 \Xi_N^{(\lambda)}
 :=\frac{1}{2N^{3/2}}
 \frac{\Gamma((N+1)/2)}{\Gamma((N+\lambda)/2)}J_N^{(\lambda)},
\end{equation}
where
\begin{equation}\label{eq:J-general}
 J_N^{(\lambda)}
 :={\int_0^\infty x^{1/2}L_{N-1}^{(\lambda)}(x)
 \left\{\Phi_{1,N}^{(\lambda)}(x)
 -\Phi_{2,N}^{(\lambda)}(x)\right\}\dd x}.
\end{equation}
\end{proposition}

\begin{proof}
We spell out the normalization because this is the only point in the proof
where the two Laguerre scalings have to be matched.  Write the standard
finite-$N$ LOE one-point formula cited above in the Livan--Vivo normalization.  Their unnormalized function
$R^{\mathrm{LV}}_{N,1,\nu}$ has integral $N$, their Laguerre variable is
$y$, and the corresponding expression \cite[Eq.~(16)]{LivanVivo} reads
\[
 R^{\mathrm{LV}}_{N,1,\nu}(y)=2R^{\mathrm{LV}}_{N,2,\nu}(2y)
 -\frac{\Gamma((N+1)/2)}{\Gamma((N+\nu)/2)}
 L_{N-1}^{(\nu)}(2y)\{\phi_1(y)-\phi_2(y)\}.
\]
Set $\nu=\lambda$.  For the $\beta=1$ ensemble,
$p_{N,\lambda}^{\mathbb R}(x)=R^{\mathrm{LV}}_{N,1,\lambda}(x/2)/(2N)$;
the factor $1/2$ is the Jacobian of $x=2y$.  The same rescaling converts the
$\beta=2$ Laguerre weight into $x^\lambda\e^{-x}$, so
$R^{\mathrm{LV}}_{N,2,\lambda}(x)/N=p_{N,\lambda}^{\mathbb C}(x)$ in our
complex normalization.  Substitution of $y=x/2$ gives
\eqref{eq:density-general}: the Laguerre factor becomes
$L_{N-1}^{(\lambda)}(x)$ and $\phi_i(x/2)$ becomes exactly
\eqref{eq:Phi1}--\eqref{eq:Phi2}.  Integration against
$N^{-1/2}x^{1/2}$ proves
\eqref{eq:alpha-general-correction}--\eqref{eq:J-general}.
\end{proof}

Taking the difference between dimensions $N$ and $N+1$ in
\eqref{eq:alpha-general-correction} gives
\begin{equation}\label{eq:real-decrement-general}
 \alpha_{\mathbb R}^{(\lambda)}(N)
 -\alpha_{\mathbb R}^{(\lambda)}(N+1)
 =\Gamma_{N,\lambda}
 -\left(\Xi_N^{(\lambda)}-\Xi_{N+1}^{(\lambda)}\right),
\end{equation}
where $\Gamma_{N,\lambda}>0$ by Paper~I~\cite[Corollary 3.2]{HutnikLUE}.  The sign of the
real decrement is therefore governed by one comparison: the unitary decrement
against the one-step loss of the orthogonal correction.

\subsection{A positive diagonal representation}

The correction in Proposition~\ref{prop:general-correction} is explicit, but
its sign is hidden by the incomplete-gamma term in $\Phi_{2,N}^{(\lambda)}$.
Two observations make the sign visible.  Shifting the Laguerre parameter by
$1/2$ turns the mixed integrals along a diagonal into a convolution of explicit
connection coefficients, and this fixes their sign.  The remaining
incomplete-gamma term is then identified, by Abel completion, with the missing
tail of the same diagonal expansion.

Define
\begin{equation*}
 \mathcal Q_\lambda(r,s)
 :=\int_0^\infty x^{\lambda+1/2}\e^{-x}
 L_r^{(\lambda)}(x)L_s^{(\lambda)}(x)\dd x.
\end{equation*}

\medskip
\noindent\textit{The diagonal kernel.}\quad
Let
\begin{equation*}
 (1-z)^{1/2}=\sum_{j=0}^\infty\sigma_jz^j,
 \qquad b_j:=-\sigma_j>0\quad(j\ge1),
\end{equation*}
and put
\begin{equation*}
 h_r^{(\lambda)}
 :=\int_0^\infty x^{\lambda+1/2}\e^{-x}
 \left(L_r^{(\lambda+1/2)}(x)\right)^2\dd x
 =\frac{\Gamma(r+\lambda+3/2)}{r!}.
\end{equation*}
The connection formula
\[
 L_r^{(\lambda)}(x)
 =\sum_{j=0}^r\sigma_jL_{r-j}^{(\lambda+1/2)}(x)
\]
gives
\begin{equation}\label{eq:Q-convolution}
 \mathcal Q_\lambda(r,s)
 =\sum_{j=0}^{\min\{r,s\}}
 \sigma_{r-j}\sigma_{s-j}h_j^{(\lambda)}.
\end{equation}
For $d\ge1$, set
\begin{equation*}
 \varkappa_d:=\frac{1}{\pi(d^2-1/4)}.
\end{equation*}
\begin{lemma}[Two identities for the connection coefficients]
\label{lem:b-identities}
For every $d\ge1$,
\begin{align}
 b_d-\sum_{t=1}^\infty b_tb_{t+d}&=\varkappa_d,
 \label{eq:b-convolution}\\
 \theta_d:=\sum_{t=1}^\infty t b_tb_{t+d}
 &=\frac1{\pi(2d+1)}.\notag
\end{align}
In particular,
\[
 \theta_{2u}=\frac1{\pi(4u+1)}<\frac1{4\pi u},
 \qquad u\ge1.
\]
\end{lemma}

\begin{proof}
Set $B(z)=1-\sqrt{1-z}=\sum_{n\ge1}b_nz^n$.  Since
$b_n=O(n^{-3/2})$, its boundary series is absolutely convergent.  The
coefficient of $z^d$ in $B(z)B(z^{-1})$ is
$\sum_{t\ge1}b_tb_{t+d}$.  On the unit circle,
\[
 B(z)\bigl(1-B(z^{-1})\bigr)
 =\sqrt{1-z^{-1}}-|1-z|.
\]
The first term has no positive Fourier coefficient.  Hence
$b_d-\sum_{t\ge1}b_tb_{t+d}$ is minus the $d$th Fourier coefficient of
$|1-\e^{\unit \theta}|=2\sin(\theta/2)$, and direct integration gives
\eqref{eq:b-convolution}.

For the weighted identity, use
\[
 b_n=\frac1\pi\int_0^1x^{n-3/2}(1-x)^{1/2}\,\dd x,
 \qquad n\ge1.
\]
Tonelli's theorem and $xB'(x)=x/(2\sqrt{1-x})$ yield
\[
 \theta_d
 =\frac1\pi\int_0^1x^{d-3/2}(1-x)^{1/2}xB'(x)\,\dd x
 =\frac1{2\pi}\int_0^1x^{d-1/2}\,\dd x
 =\frac1{\pi(2d+1)},
\] which completes the proof.
\end{proof}

\begin{proposition}[Positive diagonal kernel]\label{prop:kernel-lambda}
For every $r\ge0$, $d\ge1$, and every real $\lambda\ge0$,
\begin{equation}\label{eq:kernel-lambda}
 -\mathcal Q_\lambda(r,r+d)
 =h_r^{(\lambda)}
 \left(\varkappa_d+\mathcal E_{r,d}^{(\lambda)}\right),
\end{equation}
where
\begin{equation}\label{eq:E-lambda}
 \mathcal E_{r,d}^{(\lambda)}
 =\sum_{t=1}^r
 \left(1-\frac{h_{r-t}^{(\lambda)}}{ h_r^{(\lambda)}}\right)
 b_tb_{t+d}
 +\sum_{t=r+1}^\infty b_tb_{t+d}.
\end{equation}
Moreover, $\mathcal E_{r,d}^{(\lambda)}\ge0$ and $\mathcal E_{r+1,d}^{(\lambda)}
 \le\mathcal E_{r,d}^{(\lambda)}$.
In particular,
\begin{equation*}
 \mathcal Q_\lambda(r,r+d)<0
 \qquad(r\ge0,\ d\ge1).
\end{equation*}
\end{proposition}
\begin{proof}
Set $s=r+d$ in \eqref{eq:Q-convolution} and write $t=r-j$.  Since
$\sigma_d=-b_d$ and $\sigma_t\sigma_{t+d}=b_tb_{t+d}$ for $t\ge1$,
\[
 -\mathcal Q_\lambda(r,r+d)
 =b_dh_r^{(\lambda)}
 -\sum_{t=1}^rb_tb_{t+d}h_{r-t}^{(\lambda)}.
\]
Insert \eqref{eq:b-convolution} and divide by $h_r^{(\lambda)}$ to obtain
\eqref{eq:kernel-lambda}--\eqref{eq:E-lambda}.  Every term in
\eqref{eq:E-lambda} is nonnegative.  For fixed $t$,
\[
 \frac{h_{r-t}^{(\lambda)}}{ h_r^{(\lambda)}}
 =\prod_{j=0}^{t-1}\frac{r-j}{ r+\lambda+1/2-j}
\]
is increasing in $r$.  With
$\mathfrak p_{r,t}^{(\lambda)}:=h_{r-t}^{(\lambda)}/h_r^{(\lambda)}$ and the zero
extension $\mathfrak p_{r,r+1}^{(\lambda)}:=0$, subtracting the two representations
\eqref{eq:E-lambda} gives
\[
 \mathcal E_{r,d}^{(\lambda)}-\mathcal E_{r+1,d}^{(\lambda)}
 =\sum_{t=1}^{r+1}
 \left(\mathfrak p_{r+1,t}^{(\lambda)}-\mathfrak p_{r,t}^{(\lambda)}\right)b_tb_{t+d}\ge0.
\]
This also gives the monotonicity in $r$.
\end{proof}

Only even diagonal distances occur after the parity completion, so from now
on write
\[
 K_{r,u}^{(\lambda)}:=\varkappa_{2u}+\mathcal E_{r,2u}^{(\lambda)},
 \qquad u\ge1.
\]
By Proposition~\ref{prop:kernel-lambda}, $K_{r,u}^{(\lambda)}>0$.

\medskip
\noindent\textit{Completing the parity tail.}\quad
The finite sum in $\Phi_{1,N}^{(\lambda)}$ contains only one parity.  The Abel
argument identifies $\Phi_{2,N}^{(\lambda)}$ with the complementary infinite
tail.  Recall that $\epsilon_N=N\bmod2$.

\begin{theorem}[Exact diagonal completion]
For every $N\ge1$ and every real $\lambda\ge0$, the correction factor
$J_N^{(\lambda)}$ in \eqref{eq:J-general} satisfies
\begin{equation}\label{eq:J-tail-general}
 J_N^{(\lambda)}
 =-\sum_{m=(N+\epsilon_N)/2}^\infty
 w_{m,\epsilon_N}^{(\lambda)}
 \mathcal Q_\lambda\left(N-1,2m+1-\epsilon_N\right).
\end{equation}
\end{theorem}

\begin{proof}
Put $\omega=(\lambda+1)/2$.  Abel regularization lets us sum the parity
expansion up to its boundary value before identifying the incomplete-gamma term.
For $0<\zeta<1$, set
\[
 S_{\zeta,\epsilon_N}^{(\lambda)}(x)
 :=\sum_{m=0}^{\infty}w_{m,\epsilon_N}^{(\lambda)}
 \zeta^{2m+1-\epsilon_N}L_{2m+1-\epsilon_N}^{(\lambda)}(x).
\]
Using
\[
 w_{m,\epsilon_N}^{(\lambda)}
 =\frac{2^{1-\lambda}}{\Gamma(\omega)}
 \int_0^1t^{m-\epsilon_N/2}(1-t)^{\omega-1}\dd t
\]
together with the Laguerre generating function, and then putting $t=z^2$,
gives
\[
 S_{\zeta,\epsilon_N}^{(\lambda)}(x)
 =\frac{2^{1-\lambda}}{\Gamma(\omega)}
 \int_0^1(1-z^2)^{\omega-1}
 \bigl[\mathcal L_\lambda(\zeta z,x)
 +(-1)^{\epsilon_N+1}\mathcal L_\lambda(-\zeta z,x)\bigr]\dd z,
\]
where
\[
 \mathcal L_\lambda(z,x)=(1-z)^{-\lambda-1}
 \exp\!\left(-\frac{xz}{1-z}\right).
\]
Dominated convergence allows us to let $\zeta\uparrow1$.  The two elementary
changes of variables $y=(1+z)/(1-z)$ and $y=(1-z)/(1+z)$ give
\[
 S_{1,\epsilon_N}^{(\lambda)}(x)
 =\frac{2^{(1-\lambda)/2}x^{-\omega}\e^{x/2}}{\Gamma(\omega)}
 \bigl[\Gamma(\omega,x/2)
 +(-1)^{\epsilon_N+1}\gamma(\omega,x/2)\bigr].
\]
After multiplication by $x^\lambda\e^{-x}$ and use of
$\Gamma(\omega)=\Gamma(\omega,x/2)+\gamma(\omega,x/2)$, this is precisely
$\Phi_{2,N}^{(\lambda)}(x)$ from \eqref{eq:Phi2}.

The same limiting passage is valid after multiplication by
$x^{1/2}L_{N-1}^{(\lambda)}(x)$ and integration in $x$.  To see this, expand
the fixed Laguerre polynomial into finitely many monomials $x^j$ and integrate
first in $x$; the remaining $z$-integrand is bounded uniformly for
$0<\zeta\le1$ by a constant multiple of
$(1-z^2)^{\omega-1}(1+z)^{j+1/2}$, which is integrable.

We may now insert the Abel expansion directly into
\eqref{eq:J-general}.  The finite parity sum in $\Phi_{1,N}^{(\lambda)}$
cancels its initial segment, and the uncancelled part is exactly
\eqref{eq:J-tail-general}.  The remaining tail is an ordinary absolutely convergent series: for fixed $r$, \eqref{eq:Q-convolution} and
$b_n=O(n^{-3/2})$ give $\mathcal Q_\lambda(r,s)=O_r(s^{-3/2})$, while
$w_{m,\epsilon_N}^{(\lambda)}=O_\lambda(m^{-(\lambda+1)/2})$.  Its $m$th
term is $O_{N,\lambda}(m^{-(\lambda+4)/2})$.
\end{proof}

After the completion, the incomplete-gamma term has disappeared:
$J_N^{(\lambda)}$ is an ordinary absolutely convergent diagonal tail.
Proposition~\ref{prop:kernel-lambda} makes its sign immediate: every mixed
integral in \eqref{eq:J-tail-general} is negative and every weight
$w_{m,\epsilon_N}^{(\lambda)}$ is positive, so
$J_N^{(\lambda)}>0$ and hence $\Xi_N^{(\lambda)}>0$ for all
$N\ge1$ and all real $\lambda\ge0$.

\section{Finite-dimensional monotonicity}\label{sec:finite-comparisons}

With the correction in diagonal form, the dimension problem becomes a
comparison.  By \eqref{eq:real-decrement-general}, the real decrement is the
complex decrement minus the one-step loss of the orthogonal correction.  The three
regimes require different strengths of the same comparison.  At square
shape one diagonal of the correction is enough.  For $\lambda\ge1$ we bound
the full correction loss uniformly in the shape.  Inside $0\le\lambda\le1$
we differentiate the comparison with respect to $\lambda$.

\subsection{Estimates from the unitary problem}

We retain the half-moment notation
$\Gamma_{N,\lambda}=\Gamma_{N,1/2,\lambda}$ from Paper~I.  Only four pieces
of unitary information enter
the argument.  The square estimate comes from Abreu--Patil.  The three
shape-dependent estimates come from Paper~I and play separate roles: the first
controls the $\lambda\ge1$ comparison, the second the derivative on
$0\le\lambda\le1$, and the third the second-order transition asymptotics.

At square shape, set
\[
 \mathsf H_N:=\sum_{k=1}^N\frac1k.
\]
Then Abreu and Patil \cite[Theorem~1]{AbreuPatil} give
\begin{equation}\label{eq:LUE-square-upper}
 0<\Gamma_{N,0}
 <\frac{\mathsf H_N+6\log2-10/3}
 {8\pi[N(N+1)]^{3/2}}.
\end{equation}
For every $N\ge1$ and $\lambda\ge0$, Paper~I \cite[Corollary 5.1]{HutnikLUE} gives
\begin{equation}\label{eq:LUE-shape-reserve}
 \Gamma_{N,\lambda}
 >\mathcal R_{N,\lambda}:=
 \frac{2\lambda}{\pi[N(N+1)]^{3/2}}
 \sum_{k=1}^{N}\frac{k}{k+\lambda},
\end{equation}
where $\mathcal R_{N,\lambda}$ is strictly increasing in $\lambda$, since
$\partial_\lambda[\lambda k/(k+\lambda)]=k^2/(k+\lambda)^2>0$.
Moreover,
\begin{equation}\label{eq:LUE-derivative-reserve}
 \partial_\lambda\Gamma_{N,\lambda}
 >
 \frac{2}{\pi[N(N+1)]^{3/2}}
 \sum_{k=1}^{N}
 \left(\frac{k}{k+\lambda}\right)^2.
\end{equation}
For the asymptotic analysis, Paper~I \cite[Theorem 5.2]{HutnikLUE} also gives the uniform two-term half-moment expansion: for every $\Lambda<\infty$, uniformly for
$0\le\lambda\le\Lambda$,
\begin{equation}\label{eq:LUE-decrement-refined}
 \Gamma_{N,\lambda}
 =\frac{2\lambda}{\pi N^2}
 +\frac{1-4\lambda^2}{8\pi}\frac{\log N}{N^3}
 +O_\Lambda(N^{-3}).
\end{equation}

The square estimate \eqref{eq:LUE-square-upper} is the only square-specific
input.  The shape and derivative lower bounds and the two-term expansion
\eqref{eq:LUE-shape-reserve}--\eqref{eq:LUE-decrement-refined} all come from
Paper~I~\cite{HutnikLUE}.  We also use its locally uniform
Marchenko--Pastur convergence once, only to identify the common limiting
level.  Notice in particular that \eqref{eq:LUE-derivative-reserve} already
implies that $\lambda\mapsto\Gamma_{N,\lambda}$ is strictly increasing; no
separate unitary level expansion is needed in the real comparison.

\subsection{Variation of the diagonal terms}

Each diagonal term is a coefficient times a kernel.  The estimates below are
organized accordingly: first the kernel, then the coefficient itself, and
finally the change of that coefficient when $N$ is replaced by $N+1$.
Recall from Lemma~\ref{lem:b-identities} that
\[
 \theta_d=\frac1{\pi(2d+1)},\qquad
 \theta_{2u}<\frac1{4\pi u}.
\]
Put
\[
 \vartheta_\lambda:=\lambda+\frac{1}{2}.
\]
For $1\le t\le r$, write
\[
 \mathfrak p_{r,t}^{(\lambda)}
 :=\frac{h_{r-t}^{(\lambda)}}{ h_r^{(\lambda)}}
 =\prod_{j=0}^{t-1}\frac{r-j}{ r+\vartheta_\lambda-j},
\]
and set $\mathfrak p_{r,t}^{(\lambda)}=0$ for $t>r$.

\begin{lemma}[Kernel variation bounds]
\label{lem:kernel-variation}
For every $\lambda\ge0$, $r\ge1$, and $1\le t\le r$,
\begin{equation}\label{eq:p-r-step}
 \mathfrak p_{r+1,t}^{(\lambda)}-\mathfrak p_{r,t}^{(\lambda)}
 =\mathfrak p_{r,t}^{(\lambda)}
 \frac{\vartheta_\lambda t}{(r-t+1)(r+\vartheta_\lambda+1)}.
\end{equation}
If $\lambda\ge1$ and $r,d\ge1$, then
\begin{equation}\label{eq:E-uniform-difference}
 0\le \mathcal E_{r,d}^{(\lambda)}
 -\mathcal E_{r+1,d}^{(\lambda)}
 \le
 \frac{\vartheta_\lambda}{(r+1)(r+\vartheta_\lambda+1)}\,\theta_d.
\end{equation}
Moreover, for every $r\ge0$, $u\ge1$, and $\lambda\ge0$,
\begin{equation}\label{eq:kernel-parameter-derivative}
 0\le
 \partial_\lambda K_{r,u}^{(\lambda)}
 \le \frac{\theta_{2u}}{r+1}
 =\frac{1}{\pi(r+1)(4u+1)}.
\end{equation}
\end{lemma}

\begin{proof}
Direct division of the finite products gives \eqref{eq:p-r-step}.  If
$\lambda\ge1$, then
\[
 \mathfrak p_{r,t}^{(\lambda)}
 \le\prod_{j=0}^{t-1}\frac{r-j}{r+1-j}
 =\frac{r-t+1}{r+1},
\]
and hence, for $1\le t\le r$,
\[
 \mathfrak p_{r+1,t}^{(\lambda)}-\mathfrak p_{r,t}^{(\lambda)}
 \le
 \frac{\vartheta_\lambda t}{(r+1)(r+\vartheta_\lambda+1)}.
\]
For the new term $t=r+1$,
\[
 \mathfrak p_{r+1,r+1}^{(\lambda)}
 =\prod_{k=1}^{r+1}\frac{k}{k+\vartheta_\lambda}
 \le\frac1{r+\vartheta_\lambda+1}
 \le\frac{\vartheta_\lambda}{r+\vartheta_\lambda+1}.
\]
Subtracting the two representations \eqref{eq:E-lambda} and summing against
$b_tb_{t+d}$ gives \eqref{eq:E-uniform-difference}.

For the parameter derivative, the case $r=0$ is immediate.  If $r\ge1$,
write $\vartheta=\lambda+1/2$.  Since
\[
 K_{r,u}^{(\lambda)}
 =b_{2u}-\sum_{t=1}^r \mathfrak p_{r,t}(\vartheta)b_tb_{t+2u},
 \qquad
 \mathfrak p_{r,t}(\vartheta):=
 \prod_{j=0}^{t-1}\frac{r-j}{r+\vartheta-j},
\]
we have
\[
 -\partial_\vartheta \mathfrak p_{r,t}(\vartheta)
 =\mathfrak p_{r,t}(\vartheta)
 \sum_{j=0}^{t-1}\frac1{r+\vartheta-j}.
\]
Both factors on the right decrease with $\vartheta$, so it is enough to take
$\vartheta=1/2$.  Put $k=r-t$.  Then
\[
 \mathfrak p_{r,t}\!\left(\frac12\right)^2
 \le\prod_{n=k+1}^{r}\frac{n}{n+1}
 =\frac{k+1}{r+1},
\]
and midpoint convexity gives
\[
 \sum_{n=k+1}^{r}\frac1{n+1/2}
 \le\log\frac{r+1}{k+1}
 \le\frac{r-k}{\sqrt{(r+1)(k+1)}}.
\]
Combining the two estimates gives
\[
 0\le-\partial_\vartheta \mathfrak p_{r,t}(\vartheta)
 \le\frac{t}{r+1}.
\]
Differentiating the kernel and using Lemma~\ref{lem:b-identities} proves
\eqref{eq:kernel-parameter-derivative}.
\end{proof}

Reindexing the completed tail by its diagonal distance and using
Proposition~\ref{prop:kernel-lambda} gives the parity-free representation
\begin{equation*}
 \Xi_N^{(\lambda)}
 =\sum_{u\ge1}\mathcal A_{N,u}^{(\lambda)}
 \left(\varkappa_{2u}+\mathcal E_{N-1,2u}^{(\lambda)}\right),
\end{equation*}
where, for $N,u\ge1$,
\begin{equation}\label{eq:A-unified-gamma}
 \mathcal A_{N,u}^{(\lambda)}
 :=\frac{2^{-\lambda}}{ N^{3/2}}
 \frac{\Gamma((N+1)/2)\Gamma(N/2+u)\Gamma(N+\lambda+1/2)
 }{
 \Gamma((N+\lambda)/2)
 \Gamma((N+\lambda)/2+u+1/2)\Gamma(N)}.
\end{equation}
The parity bookkeeping ends here: it disappears from all later comparisons.  We
first control the size and shape dependence of a single coefficient; the
adjacent-dimension ratio is treated immediately
afterwards.

\begin{lemma}[Size and shape of the diagonal coefficients]
\label{lem:diagonal-coefficient-size}
For $N,u\ge1$ and $\lambda\ge0$, put
\[
\begin{aligned}
 x&=\frac{N+\lambda}{2},\qquad
 c=\frac N2,\qquad
 \omega=\frac{\lambda+1}{2},\\
 \mathcal G(x)&=\frac{\Gamma(x+1/4)\Gamma(x+3/4)}
 {\Gamma(x)\Gamma(x+1/2)},
 \qquad
 \mathfrak g(x):=\frac{\mathcal G'(x)}{\mathcal G(x)}.
\end{aligned}
\]
Then
\begin{equation}\label{eq:A-Pochhammer}
 \mathcal A_{N,u}^{(\lambda)}
 =\frac{\sqrt2}{N^{3/2}}\mathcal G(x)
 \frac{(c)_u}{(c+\omega)_u},
\end{equation}
and
\begin{equation}\label{eq:G-sqrt-bound}
 0<\mathcal G(x)\le\sqrt x.
\end{equation}
If $c\ge1$ and $\omega>0$, then
\begin{equation}\label{eq:digamma-sum}
 \sum_{u=1}^{\infty}\frac{(c)_u}{(c+\omega)_u}\frac1u
 =\psi(c+\omega)-\psi(\omega)
 \le\frac c\omega,
\end{equation}
where $\psi$ denotes the digamma function.  In particular,
\begin{equation}\label{eq:A-prefactor-global}
 \mathcal A_{N,u}^{(\lambda)}
 \le
 \frac{\sqrt{1+\lambda/N}}{N}\,
 \frac{(N/2)_u}{((N+\lambda+1)/2)_u},
\end{equation}
and, on $0\le\lambda\le1$,
\begin{equation}\label{eq:A-prefactor-strip}
 \mathcal A_{N,u}^{(\lambda)}
 \le
 \frac{\sqrt{1+1/N}}{N}\,
 \frac{(N/2)_u}{((N+\lambda+1)/2)_u}.
\end{equation}
Moreover, if $N\ge2$, then
\begin{equation}\label{eq:A-lambda-negative}
 \partial_\lambda \mathcal A_{N,u}^{(\lambda)}<0
 \qquad(\lambda\ge0).
\end{equation}
\end{lemma}

\begin{proof}
Formula \eqref{eq:A-Pochhammer} follows from
\eqref{eq:A-unified-gamma} and the duplication formula.  Log-convexity of
$\Gamma$ gives
\begin{align*}
 \Gamma(x+1/4)^2&\le\Gamma(x)\Gamma(x+1/2),\\
 \Gamma(x+3/4)^2&\le\Gamma(x+1/2)\Gamma(x+1),
\end{align*}
so $\mathcal G(x)^2\le x$.  This proves \eqref{eq:G-sqrt-bound} and the
bounds \eqref{eq:A-prefactor-global}--\eqref{eq:A-prefactor-strip}.

The identity in \eqref{eq:digamma-sum} follows from the beta integral for
$(c)_u/(c+\omega)_u$ and $\sum_{u\ge1}t^u/u=-\log(1-t)$.  For the upper
bound,
\[
 \psi(c+\omega)-\psi(\omega)
 =\int_0^\infty
 \frac{\e^{-\omega t}(1-\e^{-ct})}{1-\e^{-t}}\,\dd t
 \le\frac c\omega,
\]
because $1-\e^{-ct}\le c(1-\e^{-t})$ for $c\ge1$.

For the derivative comparisons later on, we also need the shape dependence
of a single coefficient.  The digamma difference formula gives
\begin{equation}\label{eq:g-bounds}
 \mathfrak g(x)=\int_0^\infty\frac{\e^{-xt}}{1+\e^{-t/4}}\,\dd t,
 \qquad
 \frac1{2x}<\mathfrak g(x)<\frac1{2x}+\frac1{16x^2}.
\end{equation}
Differentiating \eqref{eq:A-Pochhammer} in $\lambda$ yields
\begin{equation}\label{eq:A-lambda-log-derivative}
 2\partial_\lambda\log \mathcal A_{N,u}^{(\lambda)}
 =
 \mathfrak g(x)-\bigl[\psi(x+1/2+u)-\psi(x+1/2)\bigr].
\end{equation}
For $u\ge1$ the bracket is at least $1/(x+1/2)$.  When $N\ge2$, we have
$x\ge1$ and
\[
 \frac1{x+1/2}
 -\left(\frac1{2x}+\frac1{16x^2}\right)
 =
 \frac{16x^2-10x-1}{16x^2(2x+1)}>0.
\]
This proves \eqref{eq:A-lambda-negative}.
\end{proof}

The dependence on the dimension is most transparent through the adjacent
ratio
\[
 R_{N,u}^{(\lambda)}
 :=\frac{\mathcal A_{N+1,u}^{(\lambda)}}{\mathcal A_{N,u}^{(\lambda)}}.
\]
The following bounds on this ratio are the ones needed in the three
monotonicity arguments.

\begin{lemma}[Adjacent dimensions]\label{lem:diagonal-adjacent}
For $N,u\ge1$ and $\lambda\ge0$, put
\[
 x=\frac{N+\lambda}{2},\qquad c=\frac N2.
\]
Then
\begin{equation}\label{eq:R-factorization}
 R_{N,u}^{(\lambda)}
 =
 \left(\frac N{N+1}\right)^{3/2}
 \frac{x+1/4}{x}
 \prod_{j=0}^{u-1}
 \left(
 1+\frac{\lambda+1}
 {4(c+j)(c+j+\lambda/2+1)}
 \right),
\end{equation}
and
\begin{equation}\label{eq:R-shape-derivative}
 \partial_\lambda\log R_{N,u}^{(\lambda)}
 =
 -\frac1{8x(x+1/4)}
 +\frac14\sum_{j=0}^{u-1}
 \frac1{(x+j+1/2)(x+j+1)}.
\end{equation}
The ratio $\lambda\mapsto R_{N,u}^{(\lambda)}$ is therefore nondecreasing on
$[0,\infty)$ when $N\ge3$, and on $[1,\infty)$ when $N=2$.  At
$\lambda=1$,
\begin{equation}\label{eq:R-shape-one}
 \frac{N}{N+1}<R_{N,u}^{(1)}<1.
\end{equation}
In particular, for every $N\ge2$ and $\lambda\ge1$,
\begin{equation}\label{eq:positive-loss-dominated}
 \left(\mathcal A_{N,u}^{(\lambda)}-\mathcal A_{N+1,u}^{(\lambda)}\right)_+
 \le \mathcal A_{N,u}^{(1)}-\mathcal A_{N+1,u}^{(1)},
\end{equation}
where $y_+:=\max\{y,0\}$.  On the transition strip $0\le\lambda\le1$,
for every $N\ge3$,
\begin{equation}\label{eq:A-N-gap-positive}
 0<\mathcal A_{N,u}^{(\lambda)}-\mathcal A_{N+1,u}^{(\lambda)}
 <\frac{\mathcal A_{N,u}^{(\lambda)}}{N+1},
\end{equation}
and
\begin{equation}\label{eq:A-gap-lambda-negative}
 \partial_\lambda
 \bigl(\mathcal A_{N,u}^{(\lambda)}-\mathcal A_{N+1,u}^{(\lambda)}\bigr)<0.
\end{equation}
\end{lemma}

\begin{proof}
From \eqref{eq:A-Pochhammer},
\[
 \frac{\mathcal G(x+1/2)}{\mathcal G(x)}=\frac{x+1/4}{x},
\]
and for each $j\ge0$,
\[
 \frac{(c+j+1/2)(c+(\lambda+1)/2+j)}
 {(c+j)(c+(\lambda+1)/2+j+1/2)}
 =
 1+\frac{\lambda+1}{4(c+j)(c+j+\lambda/2+1)}.
\]
This gives \eqref{eq:R-factorization}; logarithmic differentiation gives
\eqref{eq:R-shape-derivative}.  Keeping only the $j=0$ term shows that
$\partial_\lambda R_{N,u}^{(\lambda)}\ge0$ whenever
\[
 2x(x+1/4)\ge(x+1/2)(x+1),
\]
or equivalently $x^2-x-1/2\ge0$.  This covers all $\lambda\ge0$ for
$N\ge3$ and all $\lambda\ge1$ for $N=2$.

At $\lambda=1$, the product in \eqref{eq:R-factorization} telescopes:
\[
 R_{N,u}^{(1)}
 =
 \sqrt{\frac{N}{N+1}}\,
 \frac{2N+3}{2N+2}\,
 \frac{N+2u}{N+2u+1}.
\]
The upper bound in \eqref{eq:R-shape-one} follows from
\[
 1-\frac{N(2N+3)^2}{4(N+1)^3}
 =\frac{3N+4}{4(N+1)^3}>0.
\]
For the lower bound it is enough to take $u=1$, since the last factor is
increasing in $u$.  After squaring,
\[
 \left(R_{N,1}^{(1)}\right)^2-\left(\frac{N}{N+1}\right)^2
 =
 \frac{N(13N^2+48N+36)}
 {4(N+1)^3(N+3)^2}>0.
\]
These two estimates give \eqref{eq:R-shape-one}.

We first use these bounds on the transition strip.  At $\lambda=0$,
\eqref{eq:R-factorization} becomes
\[
 R_{N,u}^{(0)}
 =
 \left(\frac{N}{N+1}\right)^{3/2}
 \frac{2N+1}{2N}
 \prod_{k=0}^{u-1}
 \frac{(N+2k+1)^2}{(N+2k)(N+2k+2)}.
\]
Every factor in the product is larger than one, and
\[
 \left[
 \sqrt{\frac{N}{N+1}}\frac{2N+1}{2N}
 \right]^2
 =1+\frac{1}{4N(N+1)}>1.
\]
So $R_{N,u}^{(0)}>N/(N+1)$.  For $N\ge3$, monotonicity in $\lambda$
and \eqref{eq:R-shape-one} now give
\[
 \frac{N}{N+1}<R_{N,u}^{(\lambda)}<1,
 \qquad 0\le\lambda\le1.
\]
Since
\[
 \mathcal A_{N,u}^{(\lambda)}-\mathcal A_{N+1,u}^{(\lambda)}
 =\mathcal A_{N,u}^{(\lambda)}(1-R_{N,u}^{(\lambda)}),
\]
this is \eqref{eq:A-N-gap-positive}.  Differentiating the same identity and
using \eqref{eq:A-lambda-negative} together with
\eqref{eq:R-shape-derivative} gives \eqref{eq:A-gap-lambda-negative}.

On the half-line $\lambda\ge1$, let $N\ge2$.  If
$R_{N,u}^{(\lambda)}\ge1$, the left-hand side of
\eqref{eq:positive-loss-dominated} is zero.  Otherwise, the decrease of
$\mathcal A_{N,u}^{(\lambda)}$ from Lemma~\ref{lem:diagonal-coefficient-size}
and the increase of $R_{N,u}^{(\lambda)}$ on $[1,\infty)$ give
\[
 \mathcal A_{N,u}^{(\lambda)}(1-R_{N,u}^{(\lambda)})
 \le
 \mathcal A_{N,u}^{(1)}(1-R_{N,u}^{(1)}),
\]
which is \eqref{eq:positive-loss-dominated}.
\end{proof}

\subsection{The square endpoint}

At square shape the argument becomes especially transparent: the first
diagonal of the orthogonal correction already dominates the entire complex
decrement.  The quantitative form of this observation is the following.

\begin{proposition}[Square comparison from the first diagonal]
\label{prop:square-one-diagonal-comparison}
For every $N\ge3$,
\begin{equation}\label{eq:square-explicit-reserve}
 \Xi_N^{(0)}-\Xi_{N+1}^{(0)}-\Gamma_{N,0}
 >
 \frac{\sqrt{N(N+1)}-\mathsf H_N-(6\log2-10/3)}
 {8\pi[N(N+1)]^{3/2}}.
\end{equation}
Consequently, for every $N\ge1$,
\[
 \Xi_N^{(0)}-\Xi_{N+1}^{(0)}-\Gamma_{N,0}
 >\frac{1}{160N^2}.
\]
\end{proposition}

\begin{proof}
Set
\[
 I_N:=\Xi_N^{(0)}-\Xi_{N+1}^{(0)}-\Gamma_{N,0}.
\]
\smallskip\noindent\emph{The first diagonal.}
For $N\ge3$, the coefficient gap in
\eqref{eq:A-N-gap-positive} and the monotonicity of the positive kernel give
\[
 \Xi_N^{(0)}-\Xi_{N+1}^{(0)}
 >\bigl(\mathcal A_{N,1}^{(0)}-\mathcal A_{N+1,1}^{(0)}\bigr)\varkappa_2.
\]
With $\mathcal G$ from Lemma~\ref{lem:diagonal-coefficient-size}, put
$\mathfrak q(x):=\mathcal G(x)/\sqrt x$.  Formula \eqref{eq:A-Pochhammer} gives
\[
 \mathcal A_{N,1}^{(0)}=\frac{\mathfrak q(N/2)}{N+1}.
\]
By \eqref{eq:g-bounds},
\[
 \frac{\dd}{\dd x}\log\mathfrak q(x)
 =\mathfrak g(x)-\frac1{2x}>0,
\]
so $\mathfrak q$ is strictly increasing.  Since
$\mathfrak q(3/2)=(15/16)\sqrt{\pi/3}>15/16$,
\[
 \mathcal A_{N,1}^{(0)}>\frac{15}{16(N+1)},\qquad N\ge3.
\]
Moreover,
\[
 \frac{\mathcal A_{N+1,1}^{(0)}}{\mathcal A_{N,1}^{(0)}}
 =\frac{(2N+1)\sqrt{N+1}}{2\sqrt N(N+2)}
 <1-\frac1{2N},
\]
because, after squaring, the last inequality is equivalent to
$4N^3-4N^2-13N+4>0$, which is positive for $N\ge3$.  We obtain
\begin{equation}\label{eq:Xi-square-lower}
 \Xi_N^{(0)}-\Xi_{N+1}^{(0)}
 >\frac1{8\pi N(N+1)},\qquad N\ge3.
\end{equation}

\smallskip\noindent\emph{From the pointwise bound to a uniform reserve.}
Put $c_0:=6\log2-10/3$.  Combining the Abreu--Patil upper bound
\eqref{eq:LUE-square-upper} with \eqref{eq:Xi-square-lower} proves the
stronger pointwise estimate \eqref{eq:square-explicit-reserve}.
We use the elementary bound $\log2<7/10$, hence $c_0<13/15$.  Since
$\mathsf H_4=25/12$, we have
$\mathsf H_4+c_0<59/20<3$, and induction gives
\[
 \mathsf H_N+c_0<\frac{3N}{4},\qquad N\ge4,
\]
since the left-hand side increases by $1/(N+1)$ while the right-hand side
increases by $3/4$.  Substitution in \eqref{eq:square-explicit-reserve}
gives
\begin{align*}
 I_N
 &>\frac{1}{32\pi\sqrt N\,(N+1)^{3/2}}
 =\frac1{32\pi N^2}\left(\frac{N}{N+1}\right)^{3/2}\\
 &\ge \frac1{32\pi N^2}\left(\frac45\right)^{3/2}
 >\frac1{160N^2},\qquad N\ge4,
\end{align*}
where the last inequality follows, for instance, from
$(4/5)^{3/2}>7/10$ and $\pi<22/7$.
The remaining dimensions $N=1,2,3$ are verified in
Appendix~\ref{app:finite-checks}.  The estimate therefore holds in every
dimension.  The stronger pointwise bound \eqref{eq:square-explicit-reserve}
has the asymptotic size $(8\pi)^{-1}N^{-2}$; thus the first diagonal alone has
a strictly positive order-$N^{-2}$ margin over the entire square unitary
decrement.
\end{proof}

\begin{proof}[Proof of Theorem~\ref{thm:square-real}]
Proposition~\ref{prop:square-one-diagonal-comparison} gives
\[
 \Xi_N^{(0)}-\Xi_{N+1}^{(0)}-\Gamma_{N,0}>\frac1{160N^2}.
\]
By \eqref{eq:real-decrement-general}, the left-hand side is exactly
$\alpha_{\mathbb R}^{(0)}(N+1)-\alpha_{\mathbb R}^{(0)}(N)$.  The convergence
to $8/(3\pi)$ follows from Proposition~\ref{prop:real-shape-basics}.
\end{proof}

\subsection{The rectangular regime}

For $\lambda\ge1$ one diagonal is no longer enough.  We instead bound the
whole correction loss uniformly in this regime.  Since the unitary lower
bound increases with $\lambda$, the final comparison can then be made at the
single value $\lambda=1$.  Put
\[
 \mathcal B_{\mathrm{even}}:=\sum_{u=1}^{\infty}b_{2u}
 =1-\frac1{\sqrt2}.
\]

\begin{lemma}[Uniform correction loss]
\label{lem:uniform-correction-loss}
For every real $\lambda\ge1$ and every $N\ge2$, we have
\begin{equation}\label{eq:uniform-Xi-parity-free}
 \Xi_N^{(\lambda)}-\Xi_{N+1}^{(\lambda)}
 <
 \frac{\sqrt{1+1/N}\,\mathcal B_{\mathrm{even}}}{(N+1)(N+2)}
 +\frac1{4\pi N^2}.
\end{equation}
\end{lemma}

\begin{proof}
From the positive-kernel representation,
\[
 0<K_{r,u}^{(\lambda)}
 =b_{2u}-\sum_{t=1}^{r}\mathfrak p_{r,t}^{(\lambda)}b_tb_{t+2u}
 \le b_{2u}.
\]
Split the correction loss as
\begin{align*}
 \Xi_N^{(\lambda)}-\Xi_{N+1}^{(\lambda)}
 ={}&\sum_{u\ge1}
 \left(\mathcal A_{N,u}^{(\lambda)}-\mathcal A_{N+1,u}^{(\lambda)}\right)
 K_{N-1,u}^{(\lambda)}\\
 &+\sum_{u\ge1}\mathcal A_{N+1,u}^{(\lambda)}
 \left(K_{N-1,u}^{(\lambda)}-K_{N,u}^{(\lambda)}\right).
\end{align*}

\smallskip\noindent\emph{Coefficient loss.}
By \eqref{eq:R-shape-one}, for every $N\ge2$,
\[
 0<\mathcal A_{N,u}^{(1)}-\mathcal A_{N+1,u}^{(1)}
 <\frac{\mathcal A_{N,u}^{(1)}}{N+1}.
\]
Together with \eqref{eq:positive-loss-dominated}, this controls the positive
coefficient loss for every real $\lambda\ge1$ and every $N\ge2$.
Moreover, \eqref{eq:A-Pochhammer} and \eqref{eq:G-sqrt-bound} yield
\[
 \mathcal A_{N,u}^{(1)}
 \le\frac{\sqrt{1+1/N}}{N+2u}.
\]
The first sum is therefore bounded by
\[
 \frac{\sqrt{1+1/N}}{N+1}
 \sum_{u\ge1}\frac{b_{2u}}{N+2u}
 \le
 \frac{\sqrt{1+1/N}\,\mathcal B_{\mathrm{even}}}{(N+1)(N+2)}.
\]

\smallskip\noindent\emph{Kernel variation.}
For the second sum put
$c=(N+1)/2$ and $\omega=(\lambda+1)/2$, and recall
$\vartheta_\lambda=\lambda+1/2$.
Lemmas~\ref{lem:kernel-variation}, \ref{lem:b-identities}, and
\ref{lem:diagonal-coefficient-size} give the upper bound
\[
 \frac{\sqrt{N+\lambda+1}}{(N+1)^{3/2}}
 \frac{\vartheta_\lambda}{N(N+\vartheta_\lambda)}
 \frac1{4\pi}
 \sum_{u\ge1}\frac{(c)_u}{(c+\omega)_u}\frac1u.
\]
Using \eqref{eq:digamma-sum}, it is enough to check
\[
 \frac{\vartheta_\lambda}{\lambda+1}
 \frac{N}{N+\vartheta_\lambda}\sqrt{\frac{N+\lambda+1}{N+1}}<1.
\]
The first factor is strictly smaller than one, while the remaining two are
controlled by
\[
 (N+\vartheta_\lambda)^2(N+1)-N^2(N+\lambda+1)>0.
\]
Together the coefficient and kernel estimates give
\eqref{eq:uniform-Xi-parity-free}.
\end{proof}

\begin{proof}[Proof of Theorem~\ref{thm:main}]
\smallskip\noindent\emph{Dimensions $N\ge2$.}
The right-hand side of \eqref{eq:uniform-Xi-parity-free} is independent of
$\lambda$, while the unitary lower bound \eqref{eq:LUE-shape-reserve} is
increasing in $\lambda$.  It is therefore enough to compare at $\lambda=1$.
For $N\ge3$,
\[
 \sum_{k=1}^N\frac{k}{k+1}
 \ge \frac12+\frac23+(N-2)\frac34
 =\frac{9N-4}{12},
\]
and $(1+1/N)^{3/2}<1+2/N$.  Hence
\begin{equation}\label{eq:R1-parity-free-lower}
 N^2\mathcal R_{N,1}
 >\frac{9N-4}{6\pi(N+2)}.
\end{equation}
On the correction side, $\mathcal B_{\mathrm{even}}<3/10$ and
$\sqrt{1+1/N}<1+1/(2N)$ give
\[
 N^2\left(\Xi_N^{(\lambda)}-\Xi_{N+1}^{(\lambda)}\right)
 <\frac{3N(2N+1)}{20(N+1)(N+2)}+\frac1{4\pi}.
\]
Using $\pi<22/7$, the difference between the lower bound in
\eqref{eq:R1-parity-free-lower} and this upper bound, after multiplication
by $\pi$, is positive because it reduces to
\[
 \frac{129N^2-163N-490}{420(N+1)(N+2)}>0,
 \qquad N\ge3.
\]

The case $N=2$ needs a separate numerical comparison.  Lemma~\ref{lem:uniform-correction-loss}
gives
\[
 \Xi_2^{(\lambda)}-\Xi_3^{(\lambda)}
 <\frac{\sqrt{3/2}\,\mathcal B_{\mathrm{even}}}{12}+\frac1{16\pi}.
\]
At $\lambda=1$ the unitary lower bound is
\[
 \mathcal R_{2,1}=\frac7{18\pi\sqrt6}.
\]
Since $\mathcal B_{\mathrm{even}}=1-1/\sqrt2<5/17$, together with
$\sqrt{3/2}<49/40$, $\pi<22/7$, and $\sqrt6<49/20$, we get
\[
 \pi\left(\frac{\sqrt{3/2}\,\mathcal B_{\mathrm{even}}}{12}+\frac1{16\pi}\right)
 <\frac8{51}<\frac{10}{63}<\frac7{18\sqrt6}.
\]
The unitary lower bound also dominates the correction loss for $N=2$.
Equation~\eqref{eq:real-decrement-general} now gives
$\DReal_N(\lambda)>0$ for every $N\ge2$ and every $\lambda\ge1$.

\smallskip\noindent\emph{Dimension $N=1$.}
The remaining case $N=1$ has a simpler direct proof, so there is no reason to
force the diagonal estimates down to the smallest dimension.  Put
\[
 x:=\frac{\lambda+1}{2}\ge1.
\]
The one-dimensional $\beta=1$ Laguerre density is a gamma density, so
\[
 \alpha_{\mathbb R}^{(\lambda)}(1)
 =\sqrt2\,\frac{\Gamma(x+1/2)}{\Gamma(x)}.
\]
The homogeneity of \eqref{eq:LOE-jpd} under $x_i=T y_i$ shows that
$T=\sum_i x_i$ has gamma law with shape $N(N+\lambda)/2$ and scale $2$.
For $N=2$ this gives the normalized one-point first moment
$\int t\,p_{2,\lambda}^{\mathbb R}(t)\dd t=\lambda+2$; Jensen's inequality therefore gives
\[
 \alpha_{\mathbb R}^{(\lambda)}(2)
 \le\sqrt{\frac{\lambda+2}{2}}
 =\sqrt{x+\frac12}.
\]
The comparison reduces to the elementary Wallis-ratio bound
\begin{equation}\label{eq:gamma-quarter-lower}
 \frac{\Gamma(x+1/2)}{\Gamma(x)}>\sqrt{x-\frac14},
 \qquad x>\frac14.
\end{equation}
To verify this bound, define
\[
 \rho(x):=\log\frac{\Gamma(x+1/2)}{\Gamma(x)}
 -\frac12\log\left(x-\frac14\right).
\]
The standard integral representation of the digamma difference gives
\begin{align*}
 \rho'(x)
 &=\int_0^\infty \e^{-xt}
 \left(\frac1{1+\e^{-t/2}}-\frac12\e^{t/4}\right)\dd t<0,
\end{align*}
 because
$\frac12(\e^{t/4}+\e^{-t/4})>1$ for $t>0$.  Stirling's formula gives
$\rho(x)\to0$ as $x\to\infty$, so \eqref{eq:gamma-quarter-lower} follows.  For
$x\ge1$ we may therefore write
\[
 \alpha_{\mathbb R}^{(\lambda)}(1)
 >\sqrt{2x-\frac12}
 \ge\sqrt{x+\frac12}
 \ge\alpha_{\mathbb R}^{(\lambda)}(2),
\]
where the middle inequality is an equality only at $x=1$ and the first one
is always strict.  Thus $\DReal_1(\lambda)>0$ for every real
$\lambda\ge1$.  The convergence to $8/(3\pi)$ is
Proposition~\ref{prop:real-shape-basics}.
\end{proof}

\subsection{Monotonicity across the transition strip}

The endpoint results show that the real decrement changes sign somewhere in
$0<\lambda<1$.  Uniqueness will follow if its shape derivative stays
positive.  We prove this by comparing the derivative of the orthogonal loss
with the explicit positive derivative reserve from the complex problem.  Put
\[
 \mathcal X_N(\lambda):=\Xi_N^{(\lambda)}-\Xi_{N+1}^{(\lambda)}.
\]
Then
\[
 \partial_\lambda\DReal_N(\lambda)
 =
 \partial_\lambda\Gamma_{N,\lambda}-\mathcal X_N'(\lambda).
\]
Whenever no argument is displayed below, a prime denotes
$\partial_\lambda$ at the current value of $\lambda$.
The diagonal series can be differentiated term by term on compact shape
intervals.  For
fixed $N$ and $0\le\lambda\le\Lambda$, the Pochhammer representation and
\eqref{eq:A-lambda-log-derivative} give
\[
 |\partial_\lambda \mathcal A_{N,u}^{(\lambda)}|
 \le C_{N,\Lambda}(1+\log(u+1))\mathcal A_{N,u}^{(\lambda)},
\]
while the finite-product kernel formula gives
\[
 K_{N-1,u}^{(\lambda)}+
 |\partial_\lambda K_{N-1,u}^{(\lambda)}|
 \le C_{N,\Lambda}b_{2u}.
\]
Since $\sum_{u\ge1}(1+\log(u+1))b_{2u}<\infty$, the diagonal series and
its first parameter derivative converge normally on compact shape intervals.

After differentiation the individual diagonal terms no longer have a common
sign, so a termwise sign argument for $\mathcal X_N'$ is unavailable.  None
is needed: negative terms only help, and it suffices to control the terms that
can contribute positively and compare them with the explicit reserve for
$\partial_\lambda\Gamma_{N,\lambda}$.

Lemmas~\ref{lem:diagonal-coefficient-size}
and~\ref{lem:diagonal-adjacent} supply all coefficient estimates needed
below: one controls size and the shape derivative, the other the
adjacent-dimension gap.  Only the differentiated kernel needs one additional
tail estimate.
Put
\[
 \mathcal B_t:=\sum_{u\ge1}b_{t+2u}.
\]

\begin{lemma}[Tail estimate for the connection coefficients]
\label{lem:bB-tail}
For every $t\ge1$,
\begin{equation}\label{eq:bB-tail}
 b_t\mathcal B_t
 \le\frac1{4\pi(t^2-1/4)}
 \le\frac1{3\pi t^2}.
\end{equation}
\end{lemma}

\begin{proof}
The beta-integral representation
\[
 b_n=\frac1\pi\int_0^1x^{n-3/2}(1-x)^{1/2}\,\dd x
\]
gives
\[
 \mathcal B_t
 =
 \frac1\pi\int_0^1
 \frac{x^{t+1/2}}{\sqrt{1-x}(1+x)}\,\dd x.
\]
Since $1+x\ge2\sqrt x$,
\[
 \mathcal B_t
 \le
 \frac{\Gamma(t+1)}{2\sqrt\pi\,\Gamma(t+3/2)}.
\]
Multiplication by
$b_t=\Gamma(t-1/2)/(2\sqrt\pi\,\Gamma(t+1))$
gives the first bound in \eqref{eq:bB-tail}; the second is equivalent to
$t^2\ge1$.
\end{proof}

\begin{lemma}[Derivative bound for the orthogonal loss]
\label{lem:Xi-loss-derivative-upper}
For $N\ge3$ and $0\le\lambda\le1$,
\begin{equation}\label{eq:Xi-derivative-simple-upper}
 \mathcal X_N'(\lambda)
 <\frac{\sqrt{1+1/N}}{4\pi N(N+1)}
 +\frac{\sqrt{1+1/(N+1)}}{3\pi(N+1)^2}.
\end{equation}
\end{lemma}

\begin{proof}
Differentiate the exact loss split and denote the four resulting sums by
$T_1,T_2,T_3,T_4$:
\begin{align*}
 \mathcal X_N'={}&
 \underbrace{\sum_{u\ge1}(\mathcal A_{N,u}'-\mathcal A_{N+1,u}')
 K_{N-1,u}}_{T_1}
 +\underbrace{\sum_{u\ge1}(\mathcal A_{N,u}-\mathcal A_{N+1,u})
 K_{N-1,u}'}_{T_2}\\
 &+\underbrace{\sum_{u\ge1}\mathcal A_{N+1,u}'
 (K_{N-1,u}-K_{N,u})}_{T_3}
 +\underbrace{\sum_{u\ge1}\mathcal A_{N+1,u}
 (K_{N-1,u}'-K_{N,u}')}_{T_4}.
\end{align*}
All quantities are evaluated at the same $\lambda$.  The coefficient estimates in
Lemmas~\ref{lem:diagonal-coefficient-size}
and~\ref{lem:diagonal-adjacent} show at once
that
\[
 T_1\le0,\qquad T_3\le0.
\]
Only $T_2$ and the positive part of $T_4$ can therefore contribute to the
upper bound.

For $T_2$, combine \eqref{eq:A-N-gap-positive},
\eqref{eq:kernel-parameter-derivative}, and
\eqref{eq:A-prefactor-strip}.  With
$c=N/2$ and $\omega=(\lambda+1)/2$,
\[
 \sum_{u\ge1}\frac{(c)_u}{(c+\omega)_u}\frac1u
 \le\frac c\omega=\frac{N}{\lambda+1}\le N
\]
by \eqref{eq:digamma-sum}.  Since
$\theta_{2u}<1/(4\pi u)$, this gives
\[
 T_2<
 \frac{\sqrt{1+1/N}}{4\pi N(N+1)}.
\]

For $T_4$, put $\vartheta=\lambda+1/2$ and
\[
 \Delta \mathfrak p_{N,t}
 :=\mathfrak p_{N,t}(\vartheta)-\mathfrak p_{N-1,t}(\vartheta),
 \qquad 1\le t\le N,
\]
with $\mathfrak p_{N-1,N}(\vartheta)=0$.  For $t<N$,
\[
 \Delta \mathfrak p_{N,t}
 =\mathfrak p_{N-1,t}(\vartheta)
 \frac{\vartheta t}{(N-t)(N+\vartheta)}.
\]
After differentiation, the term containing
$\partial_\vartheta\mathfrak p_{N-1,t}$ is nonpositive, and therefore
\[
 (\Delta \mathfrak p_{N,t}')_+
 \le \mathfrak p_{N-1,t}(\vartheta)
 \frac{Nt}{(N-t)(N+\vartheta)^2}.
\]
The new term $t=N$ is harmless because
$\Delta \mathfrak p_{N,N}'=\mathfrak p_{N,N}'<0$.
Using Lemma~\ref{lem:bB-tail}, $\mathfrak p_{N-1,t}\le1$, and
$\mathsf H_{N-1}\le N/2$, we obtain
\begin{align*}
 T_4
 &\le
 \frac{2\sqrt{1+1/(N+1)}\,\mathsf H_{N-1}}
 {3\pi(N+1)(N+\vartheta)^2}\le
 \frac{\sqrt{1+1/(N+1)}\,N}
 {3\pi(N+1)(N+\vartheta)^2}
 <\frac{\sqrt{1+1/(N+1)}}{3\pi(N+1)^2},
\end{align*}
where the last step uses
$N(N+1)<(N+1/2)^2\le(N+\vartheta)^2$.
Together with $T_1,T_3\le0$, this proves
\eqref{eq:Xi-derivative-simple-upper}.
\end{proof}

Set
\[
 \mathfrak v_N:=
 \frac{2}{\pi[N(N+1)]^{3/2}}
 \sum_{k=1}^{N}\left(\frac{k}{k+1}\right)^2.
\]
By \eqref{eq:LUE-derivative-reserve},
$\partial_\lambda\Gamma_{N,\lambda}>\mathfrak v_N$ throughout the transition
strip.  Thus, the remaining task is to show that the orthogonal derivative
stays below the same number.  The uniform estimate starts at $N=3$; the two
smaller dimensions require separate finite checks.

\begin{proposition}[Derivative comparison for $N\ge3$]
\label{prop:master-derivative-comparison}
For every $N\ge3$ and $0\le\lambda\le1$,
\[
 \mathcal X_N'(\lambda)<\mathfrak v_N.
\]
\end{proposition}

\begin{proof}
A crude lower bound is already enough:
\begin{equation}\label{eq:v-recurrence-simple}
 \mathfrak v_N\ge
 \frac{16N-7}{18\pi[N(N+1)]^{3/2}},
\end{equation}
because
\[
 \sum_{k=1}^{N}\left(\frac{k}{k+1}\right)^2
 \ge \frac14+(N-1)\frac49=\frac{16N-7}{36}.
\]
On the other hand, Lemma~\ref{lem:Xi-loss-derivative-upper} and
\[
 \sqrt{1+\frac1N}<\frac76,
 \qquad
 \sqrt{1+\frac1{N+1}}<\frac98
\]
give
\[
 \mathcal X_N'(\lambda)
 <\frac{7}{24\pi N(N+1)}
 +\frac{3}{8\pi(N+1)^2}.
\]
Using $\sqrt{N(N+1)}<N+1/2$ in \eqref{eq:v-recurrence-simple}, the difference
between the resulting lower bound for $\mathfrak v_N$ and the last upper
bound, after multiplication by $\pi N(N+1)(N+1/2)$, is
\[
 \frac{32N^2-18N-77}{144(N+1)}>0.
\]
This separates the two bounds for every $N\ge3$.
\end{proof}

\begin{lemma}[Derivative comparison in dimensions $1$ and $2$]\label{lem:low-dimensional-derivative}
For $N=1,2$ and $0\le\lambda\le1$,
\[
 \mathcal X_N'(\lambda)<\mathfrak v_N.
\]
\end{lemma}

\begin{proof}
The finite checks are given in Appendix~\ref{app:finite-checks}.
\end{proof}

\begin{proof}[Proof of Theorem~\ref{thm:global-parameter-monotonicity}]
Proposition~\ref{prop:master-derivative-comparison} gives
$\mathcal X_N'(\lambda)<\mathfrak v_N$ for $N\ge3$, and
Lemma~\ref{lem:low-dimensional-derivative} gives the same inequality for
$N=1,2$.  Since \eqref{eq:LUE-derivative-reserve} yields
$\partial_\lambda\Gamma_{N,\lambda}>\mathfrak v_N$ in every dimension, we have,
throughout $0\le\lambda\le1$,
\[
 \partial_\lambda\DReal_N(\lambda)
 =\partial_\lambda\Gamma_{N,\lambda}-\mathcal X_N'(\lambda)>0.
\] This completes the proof.
\end{proof}

\section{The transition and its asymptotics}\label{sec:asymptotics}

The finite-dimensional argument gives a unique sign change in
$0<\lambda<1$ for every $N$.  We now turn from uniqueness to location: where
does this sign change sit when $N$ is large?  The unitary expansion is already
available from Paper~I, so only the orthogonal correction decrement has to be
expanded to second order.

Write
\[
 K_{N-1,u}^{(\lambda)}=\varkappa_{2u}
 +\mathcal E_{N-1,2u}^{(\lambda)}.
\]
With this notation, $\Xi_N^{(\lambda)}-\Xi_{N+1}^{(\lambda)}$ splits into
three contributions: the principal $\varkappa$-mass, the
coefficient-weighted remainder $\mathcal E$, and the one-step change of the
kernel.  These are treated in that order.  Each contribution has a
shape-dependent logarithmic term; the important point is that the three
shape dependences cancel when they are added.

For $u\le N$ we keep the first correction to the diagonal coefficient, while
for $u>N$ only a uniform bound is needed.  Both estimates are recorded in the
same lemma.

\begin{lemma}[First-order coefficient profile]
\label{lem:first-order-diagonal-profile}
Fix $\Lambda<\infty$.  Uniformly for $0\le\lambda\le\Lambda$, put
\[
 q_{N,u}^{(\lambda)}:=
 \frac{(N/2)_u}{((N+\lambda+1)/2)_u},
 \qquad
 P_N(\lambda):=\sqrt{\frac2N}\,
 \mathcal G\!\left(\frac{N+\lambda}{2}\right),
\]
so that $\mathcal A_{N,u}^{(\lambda)}=N^{-1}P_N(\lambda)q_{N,u}^{(\lambda)}$.
Then
\begin{equation}\label{eq:P-first-order}
 P_N(\lambda)=1+O_\Lambda(N^{-1}).
\end{equation}
For $1\le u\le N$,
\begin{align}
 q_{N,u}^{(\lambda)}
 &=1-\frac{(\lambda+1)u}{N}
 +O_\Lambda\!\left(\frac{u+u^2}{N^2}\right),
 \label{eq:q-local-expansion}\\
 N\left(1-R_{N,u}^{(\lambda)}\right)
 &=1-\frac{(\lambda+1)u}{N}
 +O_\Lambda\!\left(
 \frac1N+\frac{u+u^2}{N^2}\right),
 \label{eq:R-local-expansion}
\end{align}
and hence
\begin{align}
 \mathcal A_{N,u}^{(\lambda)}
 &=\frac1N-\frac{(\lambda+1)u}{N^2}
 +O_\Lambda\!\left(
 \frac1{N^2}+\frac{u+u^2}{N^3}\right),
 \label{eq:A-local-expansion}\\
 \mathcal A_{N,u}^{(\lambda)}-\mathcal A_{N+1,u}^{(\lambda)}
 &=\frac1{N^2}-\frac{2(\lambda+1)u}{N^3}
 +O_\Lambda\!\left(
 \frac1{N^3}+\frac{u+u^2}{N^4}\right).
 \label{eq:A-gap-local-expansion}
\end{align}
For all $u\ge1$ the global bounds
\begin{equation}\label{eq:A-global-asymptotic-bounds}
 \mathcal A_{N,u}^{(\lambda)}\le\frac{C_\Lambda}{N},
 \qquad
 \left|\mathcal A_{N,u}^{(\lambda)}-
 \mathcal A_{N+1,u}^{(\lambda)}\right|\le\frac{C_\Lambda}{N^2}
\end{equation}
hold.
\end{lemma}

\begin{proof}
The fixed-shift gamma-ratio expansion \cite{TricomiErdelyi} applied to $\mathcal G$ gives
\eqref{eq:P-first-order}, locally uniformly in $\lambda$.  With
$c=N/2$ and $\omega=(\lambda+1)/2$,
\[
 \log q_{N,u}^{(\lambda)}
 =-\sum_{j=0}^{u-1}\log\left(1+\frac{\omega}{c+j}\right).
\]
For $u\le N$,
\[
 \sum_{j=0}^{u-1}\frac1{c+j}
 =\frac{2u}{N}+O\!\left(\frac{u^2}{N^2}\right),
 \qquad
 \sum_{j=0}^{u-1}\frac1{(c+j)^2}
 =O\!\left(\frac{u}{N^2}\right).
\]
Taylor's formula for the logarithm therefore gives
\[
 \log q_{N,u}^{(\lambda)}
 =-\frac{(\lambda+1)u}{N}
 +O_\Lambda\!\left(\frac{u+u^2}{N^2}\right).
\]
Exponentiating and using
$|\e^{-y}-1+y|\le C_\Lambda y^2$ for
$0\le y\le\Lambda+1$ proves \eqref{eq:q-local-expansion}.

For the adjacent ratio, write \eqref{eq:R-factorization} as a base factor
times a product.  Uniformly for $u\le N$,
\begin{align*}
 \log\left[
 \left(\frac N{N+1}\right)^{3/2}\frac{x+1/4}{x}
 \right]
 &=-\frac1N+O_\Lambda(N^{-2}),\\
 \sum_{j=0}^{u-1}
 \log\left(
 1+\frac{\lambda+1}{4(c+j)(c+j+\lambda/2+1)}
 \right)
 &=\frac{(\lambda+1)u}{N^2}
 +O_\Lambda\!\left(\frac{u+u^2}{N^3}\right).
\end{align*}
Hence
\[
 \log R_{N,u}^{(\lambda)}
 =-\frac1N+\frac{(\lambda+1)u}{N^2}
 +O_\Lambda\!\left(
 \frac1{N^2}+\frac{u+u^2}{N^3}\right).
\]
Since the right-hand side is $O_\Lambda(N^{-1})$ for $u\le N$,
exponentiation yields \eqref{eq:R-local-expansion}.  Equations
\eqref{eq:A-local-expansion}--\eqref{eq:A-gap-local-expansion} follow from
$\mathcal A_{N,u}=N^{-1}P_Nq_{N,u}$ and
$\mathcal A_{N,u}-\mathcal A_{N+1,u}
=\mathcal A_{N,u}(1-R_{N,u})$.  For the global bounds, \eqref{eq:A-prefactor-global} gives directly
$\mathcal A_{N,u}^{(\lambda)}\le C_\Lambda/N$.
Moreover, \eqref{eq:R-factorization} implies uniformly in $u$ and
$0\le\lambda\le\Lambda$ that $|\log R_{N,u}^{(\lambda)}|\le C_\Lambda/N$:
the base factor is $1+O_\Lambda(N^{-1})$ and the logarithm of the product is
bounded by a constant times $\sum_{j\ge0}(N/2+j)^{-2}=O(N^{-1})$.  Therefore
$|1-R_{N,u}^{(\lambda)}|\le C_\Lambda/N$, and
\[
 |\mathcal A_{N,u}^{(\lambda)}-\mathcal A_{N+1,u}^{(\lambda)}|
 =\mathcal A_{N,u}^{(\lambda)}|1-R_{N,u}^{(\lambda)}|
 \le \frac{C_\Lambda}{N^2}.
\]
This gives the global bounds \eqref{eq:A-global-asymptotic-bounds}.
\end{proof}

The principal contribution can be read off directly.  Write
\[
 S_\varkappa:=\sum_{u\ge1}\varkappa_{2u}.
\]
The sum is explicit, but it is convenient to keep the notation $S_\varkappa$
until the end of the calculation.  Its first correction is harmonic.

\begin{lemma}[Principal diagonal mass]
\label{lem:principal-diagonal-mass}
For every $\Lambda<\infty$, uniformly for $0\le\lambda\le\Lambda$,
\begin{equation*}
 \sum_{u\ge1}
 \bigl(\mathcal A_{N,u}^{(\lambda)}-\mathcal A_{N+1,u}^{(\lambda)}\bigr)
 \varkappa_{2u}
 =\frac{S_\varkappa}{N^2}
 -\frac{\lambda+1}{2\pi}\frac{\log N}{N^3}
 +O_\Lambda(N^{-3}).
\end{equation*}
\end{lemma}

\begin{proof}
Since
\[
 \varkappa_{2u}
 =\frac1{4\pi u^2}+O(u^{-4}),
\]
we have
\begin{equation}\label{eq:kappa-harmonic-mass}
 \sum_{u=1}^N u\varkappa_{2u}
 =\frac1{4\pi}\log N+O(1),
 \qquad
 \sum_{u>N}\varkappa_{2u}=O(N^{-1}).
\end{equation}
For $u\le N$, insert \eqref{eq:A-gap-local-expansion}.  Its remainder,
after summation against $\varkappa_{2u}$, is $O_\Lambda(N^{-3})$, because
\[
 \sum_{u\le N}\varkappa_{2u}=O(1),\qquad
 \sum_{u\le N}u\varkappa_{2u}=O(\log N),\qquad
 \sum_{u\le N}u^2\varkappa_{2u}=O(N).
\]
For $u>N$, \eqref{eq:A-global-asymptotic-bounds} and
\eqref{eq:kappa-harmonic-mass} give an $O_\Lambda(N^{-3})$ contribution.
Replacing $\sum_{u\le N}\varkappa_{2u}$ by its full sum changes only the
$O(N^{-3})$ term.  The remaining constant is explicit:
\[
 \sum_{u\ge1}\varkappa_{2u}
 =\frac2\pi\sum_{u\ge1}
 \left(\frac1{4u-1}-\frac1{4u+1}\right)
 =\frac2\pi-\frac12=S_\varkappa
\]
by the Gregory--Leibniz series, and
\eqref{eq:kappa-harmonic-mass} gives the logarithmic coefficient.
\end{proof}

The remaining two pieces are lower order before weighting, but both contribute
to the logarithmic term in the correction decrement.  Since they are driven
by the same tail product $b_t\mathcal B_t$, we evaluate them together.

\begin{lemma}[Logarithmic remainder terms]
\label{lem:log-kernel-masses}
Fix $\Lambda<\infty$.  Uniformly for $0\le\lambda\le\Lambda$,
\begin{align}
 \sum_{u\ge1}\mathcal E_{N-1,2u}^{(\lambda)}
 &=\frac{\vartheta_\lambda}{4\pi}\frac{\log N}{N}
 +O_\Lambda(N^{-1}),
 \label{eq:E-total-refined}\\
 \sum_{u\ge1}
 \bigl(K_{N-1,u}^{(\lambda)}-K_{N,u}^{(\lambda)}\bigr)
 &=\frac{\vartheta_\lambda}{4\pi}\frac{\log N}{N^2}
 +O_\Lambda(N^{-2}).
 \label{eq:K-step-total-refined}
\end{align}
Moreover,
\begin{align}
 &\sum_{u\ge1}
 \bigl(\mathcal A_{N,u}^{(\lambda)}-\mathcal A_{N+1,u}^{(\lambda)}\bigr)
 \mathcal E_{N-1,2u}^{(\lambda)}
 =\frac{\vartheta_\lambda}{4\pi}\frac{\log N}{N^3}
 +O_\Lambda(N^{-3}),
 \label{eq:weighted-E-refined}\\
 &\sum_{u\ge1}\mathcal A_{N+1,u}^{(\lambda)}
 \bigl(K_{N-1,u}^{(\lambda)}-K_{N,u}^{(\lambda)}\bigr)
 =\frac{\vartheta_\lambda}{4\pi}\frac{\log N}{N^3}
 +O_\Lambda(N^{-3}).
 \label{eq:weighted-K-step-refined}
\end{align}
\end{lemma}

\begin{proof}
Recall $\mathcal B_t=\sum_{u\ge1}b_{t+2u}$.  The beta representation from
Lemma~\ref{lem:bB-tail} gives
\[
 \mathcal B_t
 =\frac1\pi\int_0^1
 \frac{x^{t+1/2}}{\sqrt{1-x}(1+x)}\,\dd x.
\]
Since $(1+x)^{-1}=1/2+O(1-x)$ on $[0,1]$, the beta integral and the standard
gamma-ratio expansion imply
\[
 \mathcal B_t=\frac1{2\sqrt\pi}\,t^{-1/2}+O(t^{-3/2}),
 \qquad
 b_t=\frac1{2\sqrt\pi}\,t^{-3/2}+O(t^{-5/2}).
\]
Thus
\begin{equation}\label{eq:bB-asymptotic}
 b_t\mathcal B_t=\frac1{4\pi t^2}+O(t^{-3}).
\end{equation}

\smallskip\noindent\emph{The $\mathcal E$-mass.}
Put
\[
 d_{N,t}:=1-\mathfrak p_{N-1,t}^{(\lambda)},\qquad 1\le t\le N-1.
\]
For $t\le N/2$, the same logarithmic expansion used in
Lemma~\ref{lem:first-order-diagonal-profile} gives
\begin{equation}\label{eq:d-local-expansion}
 d_{N,t}=\frac{\vartheta_\lambda t}{N}
 +O_\Lambda\!\left(\frac{t+t^2}{N^2}\right).
\end{equation}
Together with \eqref{eq:bB-asymptotic}, the leading product is
$\vartheta_\lambda/(4\pi Nt)$; this harmonic factor is the source of the logarithm.
Summing \eqref{eq:E-lambda} first in $u$ gives
\[
 \sum_{u\ge1}\mathcal E_{N-1,2u}^{(\lambda)}
 =\sum_{t=1}^{N-1}d_{N,t}b_t\mathcal B_t
 +\sum_{t=N}^{\infty}b_t\mathcal B_t.
\]
The range $t>N/2$ contributes $O(N^{-1})$ by
\eqref{eq:bB-asymptotic}; on $t\le N/2$, insert
\eqref{eq:d-local-expansion} and \eqref{eq:bB-asymptotic}.  Since
$\sum_{t\le N/2}t^{-1}=\log N+O(1)$, this proves
\eqref{eq:E-total-refined}.

\smallskip\noindent\emph{The kernel step.}
Set
\[
 \delta_{N,t}:=
 \mathfrak p_{N,t}^{(\lambda)}-
 \mathfrak p_{N-1,t}^{(\lambda)},
 \qquad 1\le t\le N,
\]
with $\mathfrak p_{N-1,N}^{(\lambda)}=0$.  From the kernel representation,
\[
 K_{N-1,u}^{(\lambda)}-K_{N,u}^{(\lambda)}
 =\sum_{t=1}^{N}\delta_{N,t}b_tb_{t+2u},
\]
and therefore
\[
 \sum_{u\ge1}\bigl(K_{N-1,u}^{(\lambda)}-K_{N,u}^{(\lambda)}\bigr)
 =\sum_{t=1}^{N}\delta_{N,t}b_t\mathcal B_t.
\]
For $t<N$ there is the exact identity
\begin{equation}\label{eq:delta-p-exact}
 \delta_{N,t}
 =\mathfrak p_{N-1,t}^{(\lambda)}
 \frac{\vartheta_\lambda t}{(N-t)(N+\vartheta_\lambda)}.
\end{equation}
For $t\le N/2$ this becomes
\begin{equation}\label{eq:delta-p-local}
 \delta_{N,t}
 =\frac{\vartheta_\lambda t}{N^2}
 +O_\Lambda\!\left(\frac{t+t^2}{N^3}\right).
\end{equation}
Now the same product \eqref{eq:bB-asymptotic} contributes
$\vartheta_\lambda/(4\pi N^2t)$, so the kernel step has the same harmonic origin but
one additional factor $N^{-1}$.
For the complementary range write $k=N-t$.  The gamma form of
$\mathfrak p_{N-1,N-k}^{(\lambda)}$, together with uniform gamma-ratio bounds
for $\vartheta_\lambda$ in compact subsets of $[1/2,\infty)$, gives
\[
 \mathfrak p_{N-1,N-k}^{(\lambda)}
 \le C_\Lambda\left(\frac{k}{N}\right)^{\vartheta_\lambda},
 \qquad 1\le k\le N/2.
\]
By \eqref{eq:delta-p-exact},
$\delta_{N,N-k}\le C_\Lambda N^{-\vartheta_\lambda}k^{\vartheta_\lambda-1}$.
Since $b_t\mathcal B_t=O(N^{-2})$ on this range, its total contribution is
$O_\Lambda(N^{-2})$.  The term $t=N$ satisfies the same bound directly from
$\mathfrak p_{N,N}^{(\lambda)}=O_\Lambda(N^{-\vartheta_\lambda})$.
Combining this with \eqref{eq:delta-p-local} and
\eqref{eq:bB-asymptotic} proves \eqref{eq:K-step-total-refined}.

\smallskip\noindent\emph{Restoring the diagonal coefficients.}
To pass from these two unweighted masses to the terms that actually occur in
the correction decrement, we may replace the diagonal coefficients by their
leading values.  The required error control follows from the preceding
estimates:
\begin{equation*}
 \sum_{t=1}^{N-1}d_{N,t}b_t+\sum_{t\ge N}b_t
 =O_\Lambda(N^{-1/2}),
 \qquad
 \sum_{t=1}^{N}\delta_{N,t}b_t
 =O_\Lambda(N^{-3/2}).
\end{equation*}
Using $b_n=O(n^{-3/2})$ and summing first in $u$ yields
\begin{align}
 \sum_{u\le N}u\,\mathcal E_{N-1,2u}^{(\lambda)}&=O_\Lambda(1),\qquad
 \sum_{u\le N}u^2\mathcal E_{N-1,2u}^{(\lambda)}=O_\Lambda(N),
 \notag\\
 \sum_{u>N}\mathcal E_{N-1,2u}^{(\lambda)}&=O_\Lambda(N^{-1}),
 \label{eq:E-moment-bounds}\\
 \sum_{u\le N}u\bigl(K_{N-1,u}^{(\lambda)}-K_{N,u}^{(\lambda)}\bigr)
 &=O_\Lambda(N^{-1}),
 \notag\\
 \sum_{u\le N}u^2\bigl(K_{N-1,u}^{(\lambda)}-K_{N,u}^{(\lambda)}\bigr)
 &=O_\Lambda(1),
 \notag\\
 \sum_{u>N}\bigl(K_{N-1,u}^{(\lambda)}-K_{N,u}^{(\lambda)}\bigr)
 &=O_\Lambda(N^{-2}).
 \label{eq:K-step-moment-bounds}
\end{align}
Here we use, for $m=1,2$,
$\sum_{u\le N}u^m b_{t+2u}=O(N^{m-1/2})$ uniformly in $t$, while
$\sum_{u>N}b_{t+2u}=O(N^{-1/2})$.

Using \eqref{eq:A-gap-local-expansion},
\eqref{eq:A-global-asymptotic-bounds}, and
\eqref{eq:E-moment-bounds}, we obtain
\[
 \sum_{u\ge1}
 \left[
 \mathcal A_{N,u}^{(\lambda)}-\mathcal A_{N+1,u}^{(\lambda)}-\frac1{N^2}
 \right]\mathcal E_{N-1,2u}^{(\lambda)}
 =O_\Lambda(N^{-3}).
\]
Together with \eqref{eq:E-total-refined}, this gives
\eqref{eq:weighted-E-refined}.  Similarly,
\eqref{eq:A-local-expansion}, the global bound
$\mathcal A_{N+1,u}^{(\lambda)}\le C_\Lambda/N$, and
\eqref{eq:K-step-moment-bounds} imply
\[
 \sum_{u\ge1}
 \left[\mathcal A_{N+1,u}^{(\lambda)}-\frac1N\right]
 \bigl(K_{N-1,u}^{(\lambda)}-K_{N,u}^{(\lambda)}\bigr)
 =O_\Lambda(N^{-3}).
\]
Combining this with \eqref{eq:K-step-total-refined} gives
\eqref{eq:weighted-K-step-refined}.
\end{proof}

\begin{proposition}[Two-term orthogonal correction decrement]
\label{prop:Xi-decrement-limit}
For every $\Lambda<\infty$, uniformly for $0\le\lambda\le\Lambda$,
\begin{equation}\label{eq:Xi-decrement-local-uniform}
 \Xi_N^{(\lambda)}-\Xi_{N+1}^{(\lambda)}
 =\frac{S_\varkappa}{N^2}
 -\frac1{4\pi}\frac{\log N}{N^3}
 +O_\Lambda(N^{-3}),
 \qquad
 S_\varkappa=\frac2\pi-\frac12.
\end{equation}
\end{proposition}

\begin{proof}
Split the correction decrement as
\begin{align*}
 \Xi_N^{(\lambda)}-\Xi_{N+1}^{(\lambda)}
 ={}&\sum_{u\ge1}
 \bigl(\mathcal A_{N,u}^{(\lambda)}-\mathcal A_{N+1,u}^{(\lambda)}\bigr)
 \varkappa_{2u}\\
 &+\sum_{u\ge1}
 \bigl(\mathcal A_{N,u}^{(\lambda)}-\mathcal A_{N+1,u}^{(\lambda)}\bigr)
 \mathcal E_{N-1,2u}^{(\lambda)}\\
 &+\sum_{u\ge1}\mathcal A_{N+1,u}^{(\lambda)}
 \bigl(K_{N-1,u}^{(\lambda)}-K_{N,u}^{(\lambda)}\bigr).
\end{align*}
Lemma~\ref{lem:principal-diagonal-mass} gives the first line, while
Lemma~\ref{lem:log-kernel-masses} gives the last two.  Since $\vartheta_\lambda=\lambda+1/2$,
\[
 -\frac{\lambda+1}{2\pi}
 +\frac{\vartheta_\lambda}{4\pi}
 +\frac{\vartheta_\lambda}{4\pi}
 =-\frac1{4\pi}.
\]
The cancellation is the key point: although all three pieces separately
depend on the shape, their logarithmic coefficient does not.  This is
\eqref{eq:Xi-decrement-local-uniform}.

\end{proof}

\subsection{From the correction expansion to the transition}

\begin{proof}[Proof of Theorem~\ref{thm:transition-asymptotics}]
Combining Proposition~\ref{prop:Xi-decrement-limit} with the unitary expansion
\eqref{eq:LUE-decrement-refined} gives \eqref{eq:real-decrement-refined}.
The level expansion requires no new spectral calculation.  Telescoping from
the common limit in Proposition~\ref{prop:real-shape-basics} gives
\[
 \alpha_{\mathbb R}^{(\lambda)}(N)-\frac8{3\pi}
 =\sum_{j=N}^\infty \DReal_j(\lambda).
\]
The standard tail estimates
\[
 \sum_{j=N}^\infty j^{-2}=\frac1N+O(N^{-2}),\qquad
 \sum_{j=N}^\infty\frac{\log j}{j^3}
 =\frac{\log N}{2N^2}+O(N^{-2}),
\]
together with $\sum_{j=N}^\infty j^{-3}=O(N^{-2})$ give
\eqref{eq:real-level-refined}.

If $\lambda\ne\lambda_*$, the leading terms in the two expansions have the
sign of $\lambda-\lambda_*$.  At $\lambda=\lambda_*$ they vanish, and
$3-4\lambda_*^2>0$ makes the logarithmic terms positive and dominant for all
sufficiently large $N$, which gives the sign statements in the theorem.
\end{proof}

\begin{proof}[Proof of Corollary~\ref{cor:crossings}]
For the decrement, Theorem~\ref{thm:global-parameter-monotonicity} and the
endpoint theorems show that $\DReal_N$ has a unique zero $\lambda_N^{\mathrm{dec}}\in(0,1)$, and the
zero is simple because $\partial_\lambda\DReal_N>0$.  For the level,
Proposition~\ref{prop:real-shape-basics} gives strict monotonicity in
$\lambda$ and the common limiting value.  Together with the
endpoint Theorems~\ref{thm:square-real} and \ref{thm:main}, this places the
values at $\lambda=0$ and $\lambda=1$ on opposite sides of $8/(3\pi)$, so the
level crossing is unique as well.

The first-order terms in \eqref{eq:real-decrement-refined} and
\eqref{eq:real-level-refined} imply that both roots tend to $\lambda_*$.  At
the decrement root,
\[
 0=\frac2\pi\left(\lambda_N^{\mathrm{dec}}-\lambda_*\right)
 +\frac{3-4(\lambda_N^{\mathrm{dec}})^2}{8\pi}
   \frac{\log N}{N}+O(N^{-1}).
\]
The equation first gives
$\lambda_N^{\mathrm{dec}}-\lambda_*=O((\log N)/N)$.  Replacing
$(\lambda_N^{\mathrm{dec}})^2$ by $\lambda_*^2$ in the logarithmic
coefficient changes the equation by $o(N^{-1})$, which gives
\eqref{eq:lambdaN-convergence}.  The same argument applied to the level root
gives \eqref{eq:lambdaN-level-convergence}.  Subtracting the two expansions
then yields
\[
 \lambda_N^{\mathrm{lev}}-\lambda_N^{\mathrm{dec}}
 =\frac{3-4\lambda_*^2}{32}\frac{\log N}{N}+O(N^{-1})>0
\]
for all sufficiently large $N$, and also
$\lambda_N^{\mathrm{lev}}<\lambda_*$ eventually.
\end{proof}

\section{Conclusion}

The real Laguerre half moment exhibits a finite-dimensional shape
transition.  The square point and the region $\lambda\ge1$ lie on
opposite sides of it, and the monotonicity in the transition strip rules out
any secondary change of direction.  The crossing of the Marchenko--Pastur
level is different from the crossing of the decrement.  Both thresholds
converge to $\lambda_*=1-\pi/4$, but
\begin{align*}
 \lambda_N^{\mathrm{dec}}
 &=\lambda_*-\frac{3-4\lambda_*^2}{16}\frac{\log N}{N}+O(N^{-1}),\\
 \lambda_N^{\mathrm{lev}}
 &=\lambda_*-\frac{3-4\lambda_*^2}{32}\frac{\log N}{N}+O(N^{-1}).
\end{align*}
The interval between them is a genuine finite-size regime: the mean
is still below its limit, although the next step in dimension is already
downward.

The mechanism behind the picture is the same in all three regimes.  After the
real mean is split into a unitary half moment and a positive orthogonal
correction, the square case, the regime $\lambda\ge1$, and the transition
strip become different comparisons of one diagonal series.  The recurrence theory
of Paper~I controls the unitary part, while Abreu--Patil provide the sharper square
estimate.  The orthogonal series then
carries the symmetry-class effect.  At second order its
shape-dependent logarithmic contributions cancel, leaving the coefficient
$-1/(4\pi)$; this is what makes the separation of the two transition points
explicit.

\appendix
\section{Finite-dimensional checks}\label{app:finite-checks}

\subsection{Square endpoint: the first three dimensions}
Recall the quantity
\[
 I_N=\Xi_N^{(0)}-\Xi_{N+1}^{(0)}-\Gamma_{N,0}
\]
from the proof of Proposition~\ref{prop:square-one-diagonal-comparison}.  That
proof already
establishes the required reserve for $N\ge4$.  The remaining dimensions are
checked as follows.

For $N=3$, use $\log2<7/10$, $265/153<\sqrt3<26/15$, and $\pi<22/7$ in
\eqref{eq:square-explicit-reserve}.  Since
\[
 2\sqrt3+\frac32-6\log2
 >\frac{5845}{7650},
\]
we obtain
\[
 I_3>
 \frac{5845/7650}{192(22/7)(26/15)}
 >\frac1{1440}
 =\frac1{160\cdot3^2}.
\]
For the first two dimensions, direct integration gives
\begin{align*}
 I_1&=\frac{-4\sqrt2-\pi+2\sqrt2\,\pi}{4\sqrt\pi},\\
 I_2&=\frac{8\sqrt6+9\pi+6\sqrt3\,\pi-18\sqrt2\,\pi}
 {36\sqrt\pi}.
\end{align*}
The elementary rational bounds
$157/50<\pi<22/7$, $7/5<\sqrt2<99/70$,
$265/153<\sqrt3<26/15$, $2449/1000<\sqrt6<49/20$, and
$\sqrt\pi<9/5$ give
\[
 I_1>\frac{43}{6300}>\frac1{160},\qquad
 I_2>\frac{165751}{20241900}>\frac1{640}.
\]

\subsection{Derivative comparison in the first two dimensions}
\begin{proof}[Proof of Lemma~\ref{lem:low-dimensional-derivative}]
These two cases require no new idea; the uniform estimates used for
$N\ge3$ are simply too coarse in the first two dimensions.

\smallskip\noindent\emph{The case $N=1$.}
Here $K_{0,u}^{(\lambda)}=b_{2u}$ is independent of $\lambda$.  Put
$x=(1+\lambda)/2\in[1/2,1]$.  From
\eqref{eq:A-lambda-log-derivative} and \eqref{eq:g-bounds},
\[
 \partial_\lambda\log \mathcal A_{1,1}^{(\lambda)}\le\frac18,
 \qquad
 \partial_\lambda \mathcal A_{1,u}^{(\lambda)}\le0\quad(u\ge2).
\]
Since $\mathcal A_{1,1}^{(\lambda)}\le1/2$ and $b_2=1/8$,
\[
 (\Xi_1')_+<\frac1{128}.
\]

For $\Xi_2$, put $y=1+\lambda/2\in[1,3/2]$.  The lower bound
$\mathfrak g(y)>1/(2y)\ge1/3$ from \eqref{eq:g-bounds}, together with
\eqref{eq:A-lambda-log-derivative},
$K_{1,u}^{(\lambda)}\le b_{2u}$, and $K_{1,u}'(\lambda)\ge0$, gives
\[
 (-\Xi_2')_+\le\frac{\sqrt{3/2}}4\,\Sigma_2,
 \qquad
 \Sigma_2:=\sum_{u\ge1}\frac{(1)_u}{(3/2)_u}
 [\psi(u+3/2)-\psi(3/2)]b_{2u}.
\]
The only scalar estimate needed here is
\begin{equation}\label{eq:Sigma2-bound}
 \Sigma_2<\frac3{20}.
\end{equation}
We use the closed form
\[
 \Sigma_2=\frac{10\sqrt2-16}{9}
 +\frac{4\sqrt2}{3}\log\frac{1+\sqrt2}{2}.
\]
One way to derive it is to set
$\mathcal T(\omega)=\sum_{u\ge1}(1)_u(\omega)_u^{-1}b_{2u}$.
The series and its first derivative converge normally near $\omega=3/2$,
and $\Sigma_2=-\mathcal T'(3/2)$.  The beta representation and the even part
of $1-\sqrt{1-z}$ yield
\[
 \mathcal T(\omega)=(\omega-1)\int_0^1(1-t)^{\omega-2}
 \left[1-\frac{\sqrt{1-\sqrt t}+\sqrt{1+\sqrt t}}2\right]\dd t.
\]
Differentiating at $\omega=3/2$ and substituting $t=z^2$ gives the displayed
closed form.  Now set $a=(\sqrt2-1)/2$.  Since
$\log(1+a)<a-a^2/2+a^3/3$, it follows that
\[
 \Sigma_2<\frac{7-4\sqrt2}{9}<\frac3{20},
\]
which proves \eqref{eq:Sigma2-bound}.

Since $\sqrt{3/2}<5/4$,
\[
 \mathcal X_1'(\lambda)
 =\Xi_1'(\lambda)-\Xi_2'(\lambda)
 <\frac1{128}+\frac3{64}=\frac7{128}.
\]
On the other hand,
\[
 \mathfrak v_1=\frac1{4\pi\sqrt2}>\frac{49}{880},
 \qquad
 \frac{49}{880}-\frac7{128}=\frac7{7040}>0.
\]
Therefore $\mathcal X_1'(\lambda)<\mathfrak v_1$.

\smallskip\noindent\emph{The case $N=2$.}
Put $x=1+\lambda/2\in[1,3/2]$ and write
$R_{2,u}=\mathcal A_{3,u}/\mathcal A_{2,u}$.
For $u\ge2$, retaining only the $j=0,1$ terms in the sum in
\eqref{eq:R-shape-derivative} already gives
\[
 \partial_\lambda\log R_{2,u}
 \ge
 \frac{3(4x^4+8x^3+x^2-6x-2)}
 {2x(x+1)(x+2)(2x+1)(2x+3)(4x+1)}>0.
\]
At $\lambda=0$, the factorization \eqref{eq:R-factorization} gives
$R_{2,u}^{(0)}>2/3$, while \eqref{eq:R-shape-one} gives
$R_{2,u}^{(1)}<1$.  It follows that
\[
 \frac23<R_{2,u}<1,\qquad u\ge2,\quad 0\le\lambda\le1.
\]
For $u=1$, \eqref{eq:R-factorization} gives
\[
 R_{2,1}
 =\sqrt{\frac23}\frac{(x+1/4)(x+1/2)}{x(x+1)},
 \qquad \frac23<R_{2,1}<1,
\]
and
\[
 \partial_\lambda\log R_{2,1}
 =\frac{2x^2-2x-1}{2x(x+1)(2x+1)(4x+1)}
 \ge-\frac1{60},
\]
because the difference from $-1/60$ equals
\[
 \frac{(x-1)(8x^3+22x^2+89x+30)}
 {60x(x+1)(2x+1)(4x+1)}.
\]

Use the same four-sum decomposition of $\mathcal X_N'$ as in the proof of
Lemma~\ref{lem:Xi-loss-derivative-upper}, now with $N=2$, and call its terms
$U_1,\ldots,U_4$.  The third term is nonpositive.  For the first one,
\[
 \mathcal A_{2,u}'-\mathcal A_{3,u}'
 =\mathcal A_{2,u}'(1-R_{2,u})-\mathcal A_{2,u}R_{2,u}'.
\]
Here $\mathcal A_{2,u}'<0$ and $R_{2,u}<1$, so no positive contribution is
possible for $u\ge2$, where $R_{2,u}'\ge0$; for $u=1$ the bound
$\partial_\lambda\log R_{2,1}\ge-1/60$ gives
\[
 U_1<\frac{\mathcal A_{2,1}}{60}K_{1,1}
 <\frac1{1440},
 \qquad U_3\le0,
\]
because $\mathcal A_{2,1}\le\mathcal A_{2,1}^{(0)}<1/3$ and
$K_{1,1}\le1/8$.
For the second term, $1-R_{2,u}<1/3$ together with
Lemma~\ref{lem:kernel-variation},
\eqref{eq:A-prefactor-strip}, and \eqref{eq:digamma-sum} yields
\[
 U_2<\frac{\sqrt{3/2}}{24\pi}.
\]
For the last term, the adjacent-kernel calculation in the proof of
Lemma~\ref{lem:Xi-loss-derivative-upper} can be reused with the sharper
factor $N/(N+\vartheta_\lambda)\le4/5$.  Lemma~\ref{lem:bB-tail} then gives
\[
 U_4<\frac{16}{225\pi\sqrt3}.
\]
Combining these estimates,
\[
 \mathcal X_2'(\lambda)
 <\frac1{1440}
 +\frac{\sqrt{3/2}}{24\pi}
 +\frac{16}{225\pi\sqrt3}.
\]
The comparison value is
\[
 \mathfrak v_2=\frac{25}{108\pi\sqrt6}.
\]
After multiplication by $\pi$, the elementary bounds
\[
 \pi<\frac{22}{7},\qquad
 \sqrt{\frac32}<\frac{49}{40},\qquad
 \sqrt3>\frac{19}{11},\qquad
 \sqrt6<\frac{49}{20}
\]
give
\[
 \frac{\pi}{1440}
 +\frac{\sqrt{3/2}}{24}
 +\frac{16}{225\sqrt3}
 <\frac{20087}{212800}
 <\frac{125}{1323}
 <\frac{25}{108\sqrt6}.
\]
Hence $\mathcal X_2'(\lambda)<\mathfrak v_2$ as well.
\end{proof}

\section*{Acknowledgement}
This work was supported by the Slovak Research and Development Agency under
contract No.~APVV-25-0144. During the preparation of this work, the author used ChatGPT (OpenAI) as an auxiliary tool for literature searches, organization and language editing of the exposition, and for checking and simplifying intermediate calculations, asymptotic expansions, and parts of proofs. The author takes full responsibility for the content.

\end{document}